\documentclass[9pt]{article}
\usepackage{mathrsfs}
\usepackage{amsthm}
\usepackage{amssymb}
\usepackage{amsmath}
\usepackage{graphicx}
\usepackage{color}
\usepackage{amsfonts}
\usepackage{float}
\usepackage{cite}
\usepackage [latin1]{inputenc}
\usepackage[text={140mm,210mm},left=35mm,vmarginratio=1:1]{geometry}
\newtheorem{theorem}{Theorem}[section]

\newtheorem{lemma}[theorem]{Lemma}
\newtheorem{proposition}[theorem]{Proposition}

\newtheorem{corollary}[theorem]{Corollary}

\numberwithin{equation}{section}
\normalsize

\begin{document}
\title{\textbf{Invariance principles for additive functionals of voter models}}

\author{Xiaofeng Xue \thanks{\textbf{E-mail}: xfxue@bjtu.edu.cn \textbf{Address}: School of Mathematics and Statistics, Beijing Jiaotong University, Beijing 100044, China.}\\ Beijing Jiaotong University}

\date{}
\maketitle

\noindent {\bf Abstract:} In this paper, we prove an invariance principle for the additive functionals of a family of voter models, which include the nearest-neighbor cases on $d$-dimensional lattices for $d$ at least $5$ or on regular trees with degree at least $3$ and some long-range cases on lattices as special examples. The proof of our main result extends a martingale decomposition method, where the Kolmogorov backward equation and the duality relationship between the voter model and the coalescing random walk play the key roles.

\quad



\section{Introduction}\label{section one}

In this paper, we are concerned with the invariance principles, i.e., functional central limit theorems for the additive functionals of the voter models satisfying some common assumptions. We first review the definition of the voter model introduced in \cite{Clifford1973} and \cite{Holley1975} independently. We assume that $G$ is a countable set and $q: G\times G\rightarrow [0, +\infty)$ satisfies that
\[
\sup_{x\in G}\sum_{y\in G, y\neq x}q(x, y)<+\infty.
\]
The voter model $\{\eta_t\}_{t\geq 0}$ on $G$ with parameters $\{q(x,y)\}_{x, y\in G}$ is a continuous-time Markov process with state space $\{0, 1\}^G$ and generator $\mathcal{L}$ given by
\begin{equation}\label{equ 1.1 generator}
\mathcal{L}f(\eta)=\sum_{x\in G}\sum_{y\in G, y\neq x}q(x, y)\left(f(\eta^{x, y})-f(\eta)\right)
\end{equation}
for any $f: \{0, 1\}^G\rightarrow \mathbb{R}$ and $\eta\in \{0, 1\}^G$, where
\[
\eta^{x, y}(z)=
\begin{cases}
\eta(z) & \text{~if~}z\neq x,\\
\eta(y) & \text{~if~}z=x.
\end{cases}
\]
Intuitively speaking, in the voter model $1$ and $0$ are two opposite opinions of a topic. For each $x\in G$, $\eta(x)=1$ (resp. $\eta(x)=0$) means that $x$ is a supporter of $1$ (resp. $0$). For different $x, y\in G$, $x$ adopts the opinion of $y$ at rate $q(x, y)$. For a detailed survey of the basic properties of the voter model, readers could see Chapter 5 of \cite{Lig1985} and Part {\rm \uppercase\expandafter{\romannumeral2}} of \cite{Lig1999}.

For any $\eta\in \{0, 1\}^G$, we denote by $\mathbb{P}_\eta$ the probability measure of the voter model $\{\eta_t\}_{t\geq 0}$ starting from $\eta$. For any probability distribution $\mu$ on $\{0, 1\}^G$, we denote by $\mathbb{P}_\mu$ the probability measure of $\{\eta_t\}_{t\geq 0}$ with initial distribution $\mu$, i.e.,
\[
\mathbb{P}_\mu(\cdot)=\int_{\{0, 1\}^G}\mathbb{P}_\eta(\cdot)\mu(d\eta).
\]
We denote by $\mathbb{E}_\mu$ (resp. $\mathbb{E}_\eta$) the expectation with respect to $\mathbb{P}_\mu$ (resp. $\mathbb{P}_\eta$).

Now we introduce our assumptions of $q$ in this paper. We first assume that $q$ is symmetric, i.e.,
\begin{equation}\label{Assumption 1.2 symmetric}
q(x, y)=q(y, x)
\end{equation}
for all $x, y\in G$. Our further assumptions of $q$ are about the transition probabilities of the random walk on $G$ with transition rates $q$. We denote by $\{X_t\}_{t\geq 0}$ the continuous-time random walk on $G$ with generator $\mathcal{G}$ given by
\begin{equation}\label{equ generator of the random walk}
\mathcal{G}h(x)=\sum_{y\in G, y\neq x}q(x, y)\left(h(y)-h(x)\right)
\end{equation}
for any $h: G\rightarrow \mathbb{R}$ and $x\in G$. We denote by $\tilde{\mathbb{P}}$ the probability measure of $\{X_t\}_{t\geq 0}$ and by $\tilde{\mathbb{E}}$ the expectation with respect to $\tilde{\mathbb{P}}$. We denote by $\{p_t\}_{t\geq 0}$ the transition probabilities of $\{X_t\}_{t\geq 0}$, i.e.,
\[
p_t(x, y)=\tilde{\mathbb{P}}\left(X_t=y\Big|X_0=x\right)
\]
for any $x, y\in G$. We assume that $q$ makes
\begin{equation}\label{Assumption 1.4}
C_1:=\sup_{x, y\in G}\int_0^{+\infty}p_s(x, y)ds<+\infty,
\end{equation}
\begin{equation}\label{Assumption 1.5}
C_2:=\inf_{x\in G}\int_0^{+\infty}p_s(x, x)ds>0
\end{equation}
and
\begin{equation}\label{Assumption 1.6}
C_3:=\sup_{x, y\in G}\int_0^{+\infty}\int_0^{+\infty}p_{s_1+s_2}(x, y)ds_1ds_2<+\infty.
\end{equation}
Here we give some examples which satisfy Assumptions \eqref{Assumption 1.2 symmetric} and \eqref{Assumption 1.4}-\eqref{Assumption 1.6}.

\textbf{Example 1} \emph{Nearest-neighbor case on $\mathbb{Z}^d$ for $d\geq 5$}. When we let $G$ be the $d$-dimensional lattice $\mathbb{Z}^d$ for $d\geq 5$ and $q$ be defined as
\[
q(x, y)=1_{\{\|x-y\|_1=1\}}
\]
for any $x, y\in \mathbb{Z}^d$, where $\|x\|_1$ is the $l_1$-norm of $x$ and $1_A$ is the indicator function of the set $A$, our model reduces to the nearest-neighbor voter model on $\mathbb{Z}^d$ for $d\geq 5$. According to the local central limit theorem of the simple random walk on $\mathbb{Z}^d$ (see Theorem 2.1.3 in Chapter 2 of \cite{Lawler2010}), for any $x, y\in \mathbb{Z}^d$,
\[
p_t(x, y)\leq p_t(x, x)=\Theta(t^{-d/2}).
\]
Hence, Assumptions \eqref{Assumption 1.4}-\eqref{Assumption 1.6} all hold when $d\geq 5$.

\quad

\textbf{Example 2} \emph{Some long-range cases on $\mathbb{Z}^d$}. When we let $G$ be the $d$-dimensional lattice $\mathbb{Z}^d$ for $d\geq 1$ and $q$ be defined as
\[
q(x, y)=\|x-y\|_2^{-(d+\alpha)}
\]
for any $x, y\in \mathbb{Z}^d$, where $\|x\|_2$ is the $l_2$-norm of $x$ and $\alpha$ is a positive constant, our model reduces to the long-range voter model on $\mathbb{Z}^d$ introduced in \cite{Xue2025}. According to Lemma 1.1 of \cite{Xue2025}, which is the local central limit theorem of the long-range random walk on $\mathbb{Z}^d$, Assumptions \eqref{Assumption 1.4}-\eqref{Assumption 1.6} all hold when and only when
\[
(d, \alpha)\in \{1\}\times (0, 1/2)\bigcup \{2\}\times (0, 1)\bigcup \{3\}\times (0, 3/2) \bigcup \{4\}\times (0, 2]\bigcup \{5, 6, \ldots\}\times (0, +\infty).
\]

\quad

\textbf{Example 3} \emph{Nearest-neighbor case on regular trees}. We denote by $\mathcal{T}^d$ the regular tree where each vertex has degree $d+1$. For any $x, y\in \mathcal{T}^d$, we write $x\sim y$ when $y$ is one of $d+1$ neighbors of $x$. When we let $G=\mathcal{T}^d$ for $d\geq 2$ and $q$ be defined as
\[
q(x, y)=1_{\{x\sim y\}},
\]
our model reduces to the nearest-neighbor voter model on regular trees. By Equation (6) of \cite{Xue2026},
\[
p_t(x, y)\leq e^{-t(\sqrt{d}-1)^2}
\]
for any $x, y\in \mathcal{T}^d$ and hence Assumptions \eqref{Assumption 1.4}-\eqref{Assumption 1.6} all hold when $d\geq 2$.

\quad

Now we give the definition of the additive functional of the voter model. Throughout this paper, we are assume that $F: \{0, 1\}^G\rightarrow\mathbb{R}$ is a local function, i.e., there exist an integer $m\geq 1$, $x_1, x_2, \ldots, x_m\in G$ and $H: \{0, 1\}^m\rightarrow \mathbb{R}$ such that
\[
F(\eta)=H\left(\eta(x_1), \eta(x_2), \ldots, \eta(x_m)\right)
\]
for any $\eta\in \{0, 1\}^G$. When $\eta_0$ is distributed with probability measure $\mu$,  the centered additive functional process $\{\xi_t^F\}_{t\geq 0}$ is defined as
\begin{equation}\label{equ 1.7 additive functional}
\xi_t^F=\int_0^t\left(F(\eta_s)-\mathbb{E}_\mu F(\eta_s)\right)ds
\end{equation}
for any $t\geq 0$. When $F(\eta)=\eta(x)$ for some $x\in G$, $\xi_t^F$ reduces to the centered occupation time on $x$.

The investigations of the limit behaviors of the additive functionals of the voter models dates back to 1980s. In \cite{Cox1983}, Cox and Griffeath give the central limit theorem for the occupation time of the nearest-neighbor voter model on $\mathbb{Z}^d$ starting from a homogeneous Bernoulli product measure, which performs a dimension-dependent phase transition. Reference \cite{Cox1988} extends the main result in \cite{Cox1983} to cases where the initial distribution is a nontrivial invariant measure or the initial configuration $\eta_0$ is deterministic and has some special properties. Reference \cite{Xue2026b} extend the main result in \cite{Cox1983} to an invariance principle when $d\geq 3$ and gives a conjecture in the case $d=2$. Reference \cite{Xue2026c} further extend the main result in \cite{Cox1983} to the case where the initial Bernoulli product measure is not spatially homogeneous. Reference \cite{Xue2025} gives the invariance principles for the occupation times of the long-range voter models on $\mathbb{Z}^d$, which perform $(d, \alpha)$-dependent phase transitions. Reference \cite{Andjel1987} gives the law of large numbers for general additive functionals of the nearest-neighbor voter model on $\mathbb{Z}^d$ for $d\geq 3$ as an application of a pointwise ergodic theorem proved in the same paper. Reference \cite{Cox1988} extends the aforesaid law of large numbers to the case $d=2$. Inspired by above references, we investigate invariance principles of general additive functionals of the voter models. The main result of this paper shows that, under Assumptions \eqref{Assumption 1.2 symmetric} and \eqref{Assumption 1.4}-\eqref{Assumption 1.6}, the centered additive functional of the voter model, after proper scaling, converges weakly to a standard Brownian motion when $\eta_0$ is distributed with a homogeneous Bernoulli product measure or a nontrivial invariant measure. For a precise statement of our main result, see Section \ref{section two}.

\section{Main result} \label{section two}
In this section, we give our main result. For later use, we first introduce some notations and definitions. For given $p\in (0, 1)$, we denote by $\mu_p$ the Bernoulli product measure on $\{0, 1\}^G$ under which $\{\eta(x)\}_{x\in G}$ are independent and
\[
\mu_p\left(\eta: \eta(x)=1\right)=p
\]
for any $x\in G$. Theorem (5.8) of \cite{Holley1975} gives the following limit behavior of the voter model starting from $\mu_p$.

\begin{proposition}\label{Proposition 2.1}(Holley and Liggett, 1975). Let $p\in (0, 1)$. When $\eta_0$ is distributed with $\mu_p$, as $t\rightarrow +\infty$, $\eta_t$ converges weakly to an invariant distribution $\nu_p$ which satisfies that
\[
\mathbb{E}_{\nu_p}\eta_0(x)=p
\]
and
\[
\mathbb{E}_{\nu_p}\left(\eta_0(x)\eta_0(y)\right)=p\tilde{\mathbb{P}}\left(\tau_{x, y}<+\infty\right)+p^2\tilde{\mathbb{P}}\left(\tau_{x, y}=+\infty\right)
\]
for any $x, y\in G$, where
\[
\tau_{x, y}=\inf\{t\geq 0:~X_t^{x}=\hat{X}_t^y\}
\]
and $\{X_t^x\}_{t\geq 0}, \{\hat{X}_t^y\}_{t\geq 0}$ are independent copies of the random walk $X_t$ defined as in Section \ref{section one} with $X_0^x=x, \hat{X}_0^y=y$.
\end{proposition}
Proposition \ref{Proposition 2.1} is an application of the duality relationship between the voter model and the coalescing random walk, which we will recall later. As we have introduced in Section \ref{section one}, throughout this paper, we assume that $q$ satisfies \eqref{Assumption 1.2 symmetric} and makes \eqref{Assumption 1.4}-\eqref{Assumption 1.6} hold. Under Assumption \eqref{Assumption 1.2 symmetric}, we have that $p_t$ is symmetric. Then, under Assumption \eqref{Assumption 1.4},
\[
\tilde{\mathbb{E}}\left(\int_0^{+\infty}1_{\{X_t^x=\hat{X}_t^y\}}dt\right)=\int_0^{+\infty}p_{2t}(x, y)dt\leq \frac{C_1}{2}<+\infty
\]
by the Kolmogorov-forward equation. Consequently, according to the strong Markov property, there exist $x, y\in G$ such that $\tilde{\mathbb{P}}\left(\tau_{x, y}<+\infty\right)<1$ and hence $\nu_p$ is not a convex combination of $\delta_0$ and $\delta_1$ under Assumptions \eqref{Assumption 1.2 symmetric} and \eqref{Assumption 1.4}, where $\delta_a$ is the Dirac measure concentrated on the configuration where $\eta(x)=a$ for all $x\in G$.

Let $F:\{0, 1\}^G\rightarrow \mathbb{R}$ be the local function defined as in Section \ref{section one}. We have the following lemma about the asymptotic variance of the centered additive functional $\xi_t^F$.

\begin{lemma}\label{lemma asymptotic variance}
Let $p\in (0, 1)$, $F:\{0, 1\}^G\rightarrow \mathbb{R}$ be a local function, $\eta_0$ be distributed with $\nu_p$ and $\xi_t^F$ be defined as in \eqref{equ 1.7 additive functional} with $\mu=\nu_p$. Under Assumptions \eqref{Assumption 1.2 symmetric} and \eqref{Assumption 1.4}-\eqref{Assumption 1.6}, there exists $\sigma_F^2<+\infty$ such that
\[
\lim_{t\rightarrow+\infty}\frac{1}{t}{\rm Var}_{\nu_p}\left(\xi_t^F\right)=\sigma_F^2.
\]
\end{lemma}

The proof of Lemma \ref{lemma asymptotic variance} is given in Section \ref{section three}. Now we give the main result of this paper.

\begin{theorem}\label{theorem 2.3 main}
Let $T>0$, $p\in (0, 1)$, $F:\{0, 1\}^G\rightarrow \mathbb{R}$ be a local function, $\eta_0$ be distributed with $\mu_p$ or $\nu_p$ and $\xi_t^F$ be defined as in \eqref{equ 1.7 additive functional}. Under Assumptions \eqref{Assumption 1.2 symmetric} and \eqref{Assumption 1.4}-\eqref{Assumption 1.6}, the random process
$\left\{\frac{1}{\sqrt{N}}\xi_{tN}^F:~0\leq t\leq T\right\}$ converges weakly, with respect to the uniform topology of $C[0, T]$, to $\{\sigma_FB_t\}_{0\leq t\leq T}$ as $N\rightarrow+\infty$, where $\sigma_F^2$ is the same as that given in Lemma \ref{lemma asymptotic variance} and $\{B_t\}_{t\geq 0}$ is the standard Brownian motion starting from $0$.
\end{theorem}
Note that the expression of $\xi_t^F$ in Theorem \ref{theorem 2.3 main} depends on the distribution of $\eta_0$. When $\eta_0$ is distributed with $\mu_p$, $\xi_t^F=\int_0^t\left(F(\eta_s)-\mathbb{E}_{\mu_p}F(\eta_s)\right)ds$. When $\eta_0$ is distributed with $\nu_p$, $\xi_t^F=\int_0^t\left(F(\eta_s)-\mathbb{E}_{\nu_p}F(\eta_0)\right)ds$ since $\nu_p$ is an invariant distribution.

As corollaries of Theorem \ref{theorem 2.3 main}, we obtain the invariance principle for the additive functional of the nearest-neighbor voter model starting from $\mu_p$ or $\nu_p$ on $\mathbb{Z}^d$ for $d\geq 5$ or $\mathcal{T}^d$ for $d\geq 2$. We also obtain the invariance principle for the additive functional of the long-range voter model on lattices starting from $\mu_p$ or $\nu_p$ with the parameter $(d, \alpha)$ in
\[
\{1\}\times (0, 1/2)\bigcup \{2\}\times (0, 1)\bigcup \{3\}\times (0, 3/2) \bigcup \{4\}\times (0, 2]\bigcup \{5, 6, \ldots\}\times (0, +\infty).
\]
There are still many cases which are not covered by Theorem \ref{theorem 2.3 main} such as the nearest-neighbor case on $\mathbb{Z}^d$ for $d\leq 4$. We will work on these cases as a further investigation.

In \cite{Kipnis1986}, Kipnis and Varadhan give a functional central limit theorem for the additive functionals of Markov processes starting from a reversible distribution. Note that $\nu_p$ is an invariant distribution of the voter model $\{\eta_t\}_{t\geq 0}$ but not a reversible distribution, so Theorem \ref{theorem 2.3 main} in the case where $\eta_0$ is distributed with $\nu_p$ is not a direct corollary of the Kipnis-Varadhan theorem given in \cite{Kipnis1986}.

We prove Theorem \ref{theorem 2.3 main} in Section \ref{section three}. Our proof utilizes a martingale decomposition strategy. Roughly speaking, we choose a proper function $f: [0, +\infty)\times\{0, 1\}^G\rightarrow \mathbb{R}$ such that
\[
\left(\partial_t+\mathcal{L}\right)f(t, \eta)=-\left(F(\eta)-\mathbb{E}F(\eta_t)\right)+\text{`error term'}.
\]
Then, according to the Dynkin's martingale formula, there exists a martingale which can be decomposed as the centered additive functional $\xi_t^F$ plus an error term. Consequently, to obtain the invariance principle of the additive functional, we only need to give the invariance principle of the aforesaid martingale and verify the tightness of the additive functional. A detailed survey of this martingale decomposition method is given in \cite{Komorowski2012}. To choose a proper $f$ to execute the above strategy, we introduce the following approach. Let $\{S(t)\}_{t\geq 0}$ be the Markov semigroup of $\{\eta_t\}_{t\geq 0}$. Since $S(0)$ is the identity map, we have
\[
S(M)F(\eta)-F(\eta)=\int_0^M \frac{d}{ds}S(s)F(\eta)ds,
\]
where $M$ is a large moment depending on $N$. According to the duality relationship between the voter model and the coalescing random walk, we can give a clear expression of $S(M)F(\eta)$ and show that this term is a small error when $M$ is sufficiently large. By Kolmogorov backward equation, $\frac{d}{ds}S(s)=\mathcal{L}S(s)$. Then, we choose a centered version of $\int_0^MS(s)Fds$ as our $f$. We think that the above approach could be widely utilized in investigations of the invariance principles for the additive functionals of interacting particle systems which have relatively clear expressions of their Markov semigroups. For the mathematical details of the proof of Theorem \ref{theorem 2.3 main}, see Section \ref{section three}.

\section{Proof of Theorem \ref{theorem 2.3 main}}\label{section three}
In this section, we prove Theorem \ref{theorem 2.3 main}. As we have introduced at the end of Section \ref{section two}, our proof utilizes a martingale decomposition approach. For later use, we first review the duality relationship between the voter model and the coalescing random walk. For given integer $m\geq 1$, $x_1, x_2, \ldots, x_m\in G$ and $t_1, t_2, \ldots, t_m\geq 0$, the coalescing random walk
$\left\{\left(\mathcal{X}_t^{t_1, x_1}, \mathcal{X}_t^{t_2, x_2}, \ldots, \mathcal{X}_t^{t_m, x_m}\right)\right\}_{t\geq 0}$ evolves as follows. For all $1\leq i\leq m$, $\mathcal{X}_s^{t_i, x_i}=x_i$ when $0\leq s\leq t_i$ and $\{\mathcal{X}_{t_i+t}^{t_i, x_i}\}_{t\geq 0}$ is a copy of $\{X_t\}_{t\geq 0}$ with generator \eqref{equ generator of the random walk} starting from $x_i$. We call $X_s^{t_i, x_i}$ `frozen' when $s\leq t_i$ or `active' when $s>t_i$. The joint
distribution of $\left(\mathcal{X}_\cdot^{t_1, x_1}, \mathcal{X}_\cdot^{t_2, x_2}, \ldots, \mathcal{X}_\cdot^{t_m, x_m}\right)$ is as follows. All active walks evolve independently until collision occurs. When two active walks collide with each other, they are coalesced into one active walk. When an active walk hits a frozen walk, they are not coalesced. Note that, when $x_i=x_j$ and $t_i=t_j$ for some $i, j$, $\{\mathcal{X}_t^{t_i, x_i}\}_{t\geq 0}$ and $\{\mathcal{X}_t^{t_j, x_j}\}_{t\geq 0}$ are coalesced as one walk from beginning.

We denote by $\hat{\mathbb{P}}$ the probability measure of the coalescing random walk and by $\hat{\mathbb{E}}$ the expectation with respect to $\hat{\mathbb{P}}$. For simplicity, we write $\mathcal{X}_t^{0, x}$ as $\mathcal{X}_t^x$. For $x_1, \ldots, x_m, y_1, \ldots, y_m\in G$ and $t\geq 0$, we denote by $\beta_t\left((x_1, \ldots, x_m), (y_1, \ldots, y_m)\right)$ the term
\[
\hat{\mathbb{P}}\left(\left(\mathcal{X}_t^{x_1}, \ldots, \mathcal{X}_t^{x_m}\right)=(y_1, \ldots, y_m)\right).
\]

The following duality relationship between the voter model and the coalescing random walk is independently given in \cite{Clifford1973} and \cite{Holley1975}.

\begin{proposition}\label{proposition 3.1 dual}(Clifford and Sudbury, 1973; Holley and Liggett, 1975).
For any integer $m\geq 1$, $x_1, x_2, \ldots, x_m\in G$, $0\leq t_1, t_2, \ldots, t_m\leq s$ and given $\eta\in \{0, 1\}^G$, we have
\begin{align}\label{equ 3.1 dual}
\mathbb{P}_\eta\left(\eta_{t_i}(x_i)=1\text{~for all~}1\leq i\leq m\right)
=\hat{\mathbb{P}}\left(\mathcal{X}_s^{s-t_i, x_i}\in \mathcal{A}(\eta)\text{~for all~}1\leq i\leq m\right),
\end{align}
where $A(\eta)=\{x\in G:~\eta(x)=1\}$.
\end{proposition}
According to the graphical method introduced in \cite{Har1978}, $\mathcal{X}_u^{s-t_i, x_i}$ can be considered as the owner of the opinion $\eta_{t_i}(x_i)$ at the moment $s-u$ for any $s-t_i\leq u\leq s$, i.e., in the sense of coupling,
\begin{equation}\label{equ 3.2 coupling}
\eta_{t_i}(x_i)=\eta_{s-u}(\mathcal{X}_u^{s-t_i, x_i}).
\end{equation}
Equation \eqref{equ 3.1 dual} follows from \eqref{equ 3.2 coupling} directly.

Note that Proposition \ref{Proposition 2.1} is a corollary of Proposition \ref{proposition 3.1 dual}. By Proposition \ref{proposition 3.1 dual},
\[
\mathbb{P}_{\mu_p}\left(\eta_{t}(x_i)=1\text{~for all~}1\leq i\leq m\right)
=\hat{\mathbb{E}}p^{{\rm card}\left\{\mathcal{X}_t^{x_1}, \mathcal{X}_t^{x_2}, \ldots, \mathcal{X}_t^{x_m}\right\}},
\]
where ${\rm card}(A)$ is the cardinality of the set $A$. Consequently, the existence of $\nu_p$ follows from the fact that ${\rm card}\left\{\mathcal{X}_t^{x_1}, \mathcal{X}_t^{x_2}, \ldots, \mathcal{X}_t^{x_m}\right\}$ is decreasing in $t$.

Since each local $F: \{0, 1\}^G\rightarrow \mathbb{R}$ is a linear combination of several indicator functions $1_{\{\eta(x)=1\text{~for all~}x\in B\}}$ for some $B\subseteq G$, we have the following equivalent statement of Proposition \ref{proposition 3.1 dual}, which is more convenient to utilize in our proof.
\begin{proposition}\label{proposition 3.2 dual equivalent}(Clifford and Sudbury, 1973; Holley and Liggett, 1975).
For any integer $m\geq 1$, $x_1, x_2, \ldots, x_m\in G$, $0\leq t_1, t_2, \ldots, t_m\leq s$, given $\eta\in \{0, 1\}^G$ and $H:\{0, 1\}^m\rightarrow \mathbb{R}$, we have
\begin{align}\label{equ 3.3 dual}
&\mathbb{E}_\eta H\left(\eta_{t_1}(x_1), \eta_{t_2}(x_2)\ldots, \eta_{t_m}(x_m)\right) \notag\\
&=\hat{\mathbb{E}}H\left(\eta(\mathcal{X}_s^{s-t_1, x_1}), \eta(\mathcal{X}_s^{s-t_2, x_2})\ldots, \eta(\mathcal{X}_s^{s-t_m, x_m})\right).
\end{align}
Consequently, for any probability distribution $\mu$ on $\{0, 1\}^G$,
\begin{align*}
&\mathbb{E}_{\mu} H\left(\eta_{t_1}(x_1), \eta_{t_2}(x_2)\ldots, \eta_{t_m}(x_m)\right) \notag\\
&=\hat{\mathbb{E}}\mathbb{E}_{\mu}H\left(\eta_0(\mathcal{X}_s^{s-t_1, x_1}), \eta_0(\mathcal{X}_s^{s-t_2, x_2})\ldots, \eta_0(\mathcal{X}_s^{s-t_m, x_m})\right).
\end{align*}
\end{proposition}

Let $F: \{0, 1\}^G\rightarrow\mathbb{R}$ be a local function. As we have introduced in Section \ref{section one}, there exist integer $m\geq 1$, $x_1, \ldots, x_m\in G$ and $H: \{0, 1\}^m\rightarrow\mathbb{R}$ such that
\[
F(\eta)=H\left(\eta(x_1), \ldots, \eta(x_m)\right)
\]
for any $\eta\in \{0, 1\}^G$. From now on, for simplicity, we write $H(\eta(x_1), \ldots, \eta(x_m))$ as
\[
H(\eta, x_1, \ldots, x_m).
\]
The following lemma is crucial for us to prove Lemma \ref{lemma asymptotic variance} and Theorem \ref{theorem 2.3 main}.

\begin{lemma}\label{lemma 3.zero}
For any integer $m\geq 1$, $H_1, H_2: \{0, 1\}^m\rightarrow \mathbb{R}$,
\[
x_1, x_2, \ldots, x_m, y_1, y_2, \ldots, y_m\in G
\]
and $t, r\geq 0$, we have
\begin{align*}
&\left|{\rm Cov}_{\mu_p}\left(H_1(\eta_{t+r}, x_1, \ldots, x_m), H_2(\eta_t, y_1, \ldots, y_m)\right)\right|\\
&\leq \frac{2\|H_1\|_\infty\|H_2\|_\infty}{C_2}\sum_{i=1}^m\sum_{j=1}^m\int_0^{+\infty}p_{s+r}(x_i, y_j)ds,
\end{align*}
where $C_2$ is defined as in Assumption \eqref{Assumption 1.5} and $\|H_1\|_\infty=\max_{\iota\in \{0, 1\}^m}|H_1(\iota)|$.
\end{lemma}

\proof
By applying Proposition \ref{proposition 3.2 dual equivalent} with $t_1=t_2=\ldots=t_m=s=r$, we have
\[
\mathbb{E}_{\eta_0}H_1(\eta_r, x_1, \ldots, x_m)=\sum_{z_1, \ldots, z_m\in G}\beta_r\left((x_1, \ldots, x_m), (z_1, \ldots, z_m)\right)H_1(\eta_0, z_1, \ldots, z_m).
\]
Then, by Markov property, we have
\begin{align}\label{equ add 1}
&{\rm Cov}_{\mu_p}\left(H_1(\eta_{t+r}, x_1, \ldots, x_m), H_2(\eta_t, y_1, \ldots, y_m)\right)\notag\\
&=\sum_{z_1, \ldots, z_m\in G}\beta_r\left((x_1, \ldots, x_m), (z_1, \ldots, z_m)\right)
\Delta_{H_1H_2}(p, t, z_1, \ldots, z_m, y_1, \ldots, y_m),
\end{align}
where
\begin{align*}
\Delta_{H_1H_2}(p, t, z_1, \ldots, z_m, y_1, \ldots, y_m)
&=\mathbb{E}_{\mu_p}\left(H_1(\eta_t, z_1, \ldots, z_m)H_2(\eta_t, x_1, \ldots, x_m)\right)\\
&\text{\quad\quad}-\mathbb{E}_{\mu_p}H_1(\eta_t, z_1, \ldots, z_m)\mathbb{E}_{\mu_p}H_2(\eta_t, y_1, \ldots, y_m).
\end{align*}
Let
\[
\{\left(\mathcal{X}_t^{z_1}, \ldots, \mathcal{X}_t^{z_m}, \mathcal{X}_t^{y_1}, \ldots, \mathcal{X}_t^{y_m}\right)\}_{t\geq 0}
\]
be the (no freeze) coalescing random walk starting from $z_1, \ldots, z_m, y_1, \ldots, y_m$ and
\[
\left\{\left(\tilde{\mathcal{X}}_t^{y_1}, \ldots, \tilde{\mathcal{X}}_t^{y_m}\right)\right\}_{t\geq 0}
\]
be another coalescing random walk starting from $y_1, \ldots, y_m$ which is independent of
\[
\left\{\left(\mathcal{X}_t^{z_1}, \ldots, \mathcal{X}_t^{z_m}\right)\right\}_{t\geq 0}.
\]
According to the definition of the coalescing random walk, we can further assume that
\[
\left(\tilde{\mathcal{X}}_t^{y_1}, \ldots, \tilde{\mathcal{X}}_t^{y_m}\right)=
\left(\mathcal{X}_t^{y_1}, \ldots, \mathcal{X}_t^{y_m}\right)
\]
for all $t\leq \hat{\tau}$, where
\[
\hat{\tau}=\inf\{t:~\mathcal{X}_t^{z_i}=\mathcal{X}_t^{y_j}\text{~for some~}1\leq i, j\leq m\}.
\]
By Proposition \ref{proposition 3.2 dual equivalent},
\begin{align*}
&\mathbb{E}_{\mu_p}\left(H_2(\eta_t, y_1, \ldots, y_m)H_1(\eta_t, z_1, \ldots, z_m)\right)\\
&=\hat{\mathbb{E}}\mathbb{E}_{\mu_p}\left(H_2(\eta_0, \mathcal{X}_t^{y_1}, \ldots, \mathcal{X}_t^{y_m})
H_1(\eta_0, \mathcal{X}_t^{z_1}, \ldots, \mathcal{X}_t^{z_m})\right)
\end{align*}
and
\begin{align*}
&\mathbb{E}_{\mu_p}H_2(\eta_t, y_1, \ldots, y_m)\mathbb{E}_{\mu_p}H_1(\eta_t, z_1, \ldots, z_m)\\
&=\hat{\mathbb{E}}\mathbb{E}_{\mu_p}\left(H_2(\eta_0, \tilde{\mathcal{X}}_t^{y_1}, \ldots, \tilde{\mathcal{X}}_t^{y_m})\right)
\hat{\mathbb{E}}\mathbb{E}_{\mu_p}\left(H_1(\eta_0, \mathcal{X}_t^{z_1}, \ldots, \mathcal{X}_t^{z_m})\right)\\
&=\hat{\mathbb{E}}\left(\mathbb{E}_{\mu_p}\left(H_2(\eta_0, \tilde{\mathcal{X}}_t^{y_1}, \ldots, \tilde{\mathcal{X}}_t^{y_m})\right)
\mathbb{E}_{\mu_p}\left(H_1(\eta_0, \mathcal{X}_t^{z_1}, \ldots, \mathcal{X}_t^{z_m})\right)\right)
\end{align*}
since
$
\left\{\left(\tilde{\mathcal{X}}_t^{y_1}, \ldots, \tilde{\mathcal{X}}_t^{y_m}\right)\right\}_{t\geq 0}
$
is independent of
$
\left\{\left(\mathcal{X}_t^{z_1}, \ldots, \mathcal{X}_t^{z_m}\right)\right\}_{t\geq 0}.
$
On the event $\{\hat{\tau}>t\}$,
\begin{align*}
&\mathbb{E}_{\mu_p}\left(H_2(\eta_0, \tilde{\mathcal{X}}_t^{y_1}, \ldots, \tilde{\mathcal{X}}_t^{y_m})\right)
\mathbb{E}_{\mu_p}\left(H_1(\eta_0, \mathcal{X}_t^{z_1}, \ldots, \mathcal{X}_t^{z_m})\right)\\
&=\mathbb{E}_{\mu_p}\left(H_2(\eta_0, \mathcal{X}_t^{y_1}, \ldots, \mathcal{X}_t^{y_m})\right)
\mathbb{E}_{\mu_p}\left(H_1(\eta_0, \mathcal{X}_t^{z_1}, \ldots, \mathcal{X}_t^{z_m})\right)\\
&=\mathbb{E}_{\mu_p}
\left(H_2(\eta_0, \mathcal{X}_t^{y_1}, \ldots, \mathcal{X}_t^{y_m})H_1(\eta_0, \mathcal{X}_t^{z_1}, \ldots, \mathcal{X}_t^{z_m})\right)
\end{align*}
since $\mu_p$ is a product measure and
\[
\left\{\mathcal{X}_t^{z_1}, \ldots, \mathcal{X}_t^{z_m}\right\}
\bigcap\left\{\mathcal{X}_t^{y_1}, \ldots, \mathcal{X}_t^{y_m}\right\}=\emptyset
\]
when $\hat{\tau}>t$. Therefore,
\begin{align*}
&\left|\mathbb{E}_{\mu_p}\left(H_2(\eta_t, y_1, \ldots, y_m)H_1(\eta_t, z_1, \ldots, z_m)\right)
-\mathbb{E}_{\mu_p}H_2(\eta_t, y_1, \ldots, y_m)\mathbb{E}_{\mu_p}H_1(\eta_t, z_1, \ldots, z_m)\right|\\
&\leq 2\|H_1\|_\infty\|H_2\|_\infty\hat{\mathbb{P}}(\hat{\tau}<t)\leq 2\|H_1\|_\infty\|H_2\|_\infty\hat{\mathbb{P}}(\hat{\tau}<+\infty)
\end{align*}
for any $t\geq 0$ and consequently
\begin{equation}\label{equ add 2}
|\Delta_{H_1H_2}(p, t, z_1, \ldots, z_m, y_1, \ldots, y_m)|\leq  2\|H_1\|_\infty\|H_2\|_\infty\hat{\mathbb{P}}(\hat{\tau}<+\infty).
\end{equation}
According to the definition of $\hat{\tau}$, we have
\[
\hat{\mathbb{P}}(\hat{\tau}<+\infty)
\leq \sum_{i=1}^m\sum_{j=1}^m\tilde{\mathbb{P}}(\tau_{z_i, y_j}<+\infty),
\]
where $\tau_{x, y}$ is defined as in Proposition \ref{Proposition 2.1}. According to the strong Markov property,
\begin{align*}
&\tilde{\mathbb{E}}\int_0^{+\infty}1_{\{X_s^{z_i}=\hat{X}_s^{y_j}\}}ds\\
&=\sum_{z\in G}\tilde{\mathbb{P}}\left(\tau_{z_i, y_j}<+\infty,
X^{z_i}_{\tau_{z_i, y_j}}=\hat{X}^{y_j}_{\tau_{z_i, y_j}}=z\right)
\tilde{\mathbb{E}}\int_0^{+\infty}1_{\{X_s^z=\hat{X}_s^z\}}ds.
\end{align*}
As we have pointed out, $\tilde{\mathbb{P}}(X_s^{x}=\hat{X}_s^{y})=p_{2s}(x, y)$ under Assumption \eqref{Assumption 1.2 symmetric}. Therefore, $\tilde{\mathbb{E}}\int_0^{+\infty}1_{\{X_s^z=\hat{X}_s^z\}}ds\geq \frac{C_2}{2}$ under Assumption \eqref{Assumption 1.5}. Consequently,
\[
\tilde{\mathbb{P}}\left(\tau_{z_i, y_j}<+\infty\right)\leq \frac{1}{C_2}\int_0^{+\infty}p_s(z_i, y_j)ds
\]
and
\[
\hat{\mathbb{P}}(\hat{\tau}<+\infty)
\leq \sum_{i=1}^m\sum_{j=1}^m\frac{1}{C_2}\int_0^{+\infty}p_s(z_i, y_j)ds.
\]
For fixed $z_j$, according to the definition of $\beta_r$,
\[
\sum_{z_1,\ldots, z_{j-1}, z_{j+1}, \ldots, z_m\in G}\beta_r\left((x_1, \ldots, x_m), (z_1, \ldots, z_m)\right)
=p_r(x_j, z_j).
\]
Then, since $\sum_{z_j\in G}p_r(x_j, z_j)p_s(z_j, y_i)=p_{r+s}(x_j, y_i)$,
\begin{align*}
&\left|{\rm Cov}_{\mu_p}\left(H_1(\eta_{t+r}, x_1, \ldots, x_m), H_2(\eta_t, y_1, \ldots, y_m)\right)\right|\\
&\leq \frac{2\|H_1\|_\infty\|H_2\|_\infty}{C_2}\sum_{i=1}^m\sum_{j=1}^m\int_0^{+\infty}p_{s+r}(x_i, y_j)ds
\end{align*}
by \eqref{equ add 1} and \eqref{equ add 2}.
\qed

Now we prove Lemma \ref{lemma asymptotic variance}.

\proof[Proof of Lemma \ref{lemma asymptotic variance}]
According to the bilinear property of the covariance and the invariance of $\nu_p$,
\begin{equation}\label{equ 3.4}
{\rm Var}_{\nu_p}(\xi_t^F)=2\int_0^t\left(\int_0^s{\rm Cov}_{\nu_p}\left(F(\eta_r), F(\eta_0)\right)dr\right)ds.
\end{equation}
Consequently, to complete this proof, we only need to show that
\begin{equation}\label{equ 3.5}
\int_0^{+\infty}\left|{\rm Cov}_{\nu_p}\left(F(\eta_r), F(\eta_0)\right)\right|dr<+\infty.
\end{equation}
By \eqref{equ 3.4} and \eqref{equ 3.5}, we have
\begin{equation}\label{equ 3.6}
\sigma_F^2=2\int_0^{+\infty}{\rm Cov}_{\nu_p}\left(F(\eta_r), F(\eta_0)\right)dr.
\end{equation}
Now we verify \eqref{equ 3.5}. According to the definition of $\nu_p$,
\[
{\rm Cov}_{\nu_p}\left(F(\eta_r), F(\eta_0)\right)
=\lim_{t\rightarrow+\infty}{\rm Cov}_{\mu_p}\left(H(\eta_{t+r}, x_1, \ldots, x_m), H(\eta_t, x_1, \ldots, x_m)\right).
\]
Then, by Lemma \ref{lemma 3.zero},
\[
\left|{\rm Cov}_{\nu_p}\left(F(\eta_r), F(\eta_0)\right)\right|\leq
\frac{2\|H\|_\infty^2}{C_2}\sum_{i=1}^m\sum_{j=1}^m\int_0^{+\infty}p_{s+r}(x_i, x_j)ds
\]
and consequently \eqref{equ 3.5} follows from Assumption \eqref{Assumption 1.6}.

\qed

The following lemma shows that, when $\eta_0$ is distributed with $\mu_p$, the limit variance of $\frac{1}{\sqrt{t}}\xi_t^F$ is also $\sigma_F^2$.

\begin{lemma}\label{lemma 3.3}
Let $\eta_0$ be distributed with $\mu_p$ and $\xi_t^F$ be defined as in \eqref{equ 1.7 additive functional} with $\mu=\mu_p$. We have
\[
\lim_{t\rightarrow+\infty}{\rm Var}_{\mu_p}\left(\frac{1}{\sqrt{t}}\xi_t^F\right)=\sigma_F^2,
\]
where $\sigma_F^2$ is the same as that given in Lemma \ref{lemma asymptotic variance}.
\end{lemma}

\proof
According to the bilinear property of the covariance,
\[
{\rm Var}_{\mu_p}(\xi_t^F)=2\int_0^{t}\left(\int_0^s{\rm Cov}_{\mu_p}\left(F(\eta_s), F(\eta_{s-r})\right)dr\right)as.
\]
According to the definition of $\nu_p$, for any $r\geq 0$,
\begin{equation}\label{equ 3.10}
\lim_{s\rightarrow+\infty}{\rm Cov}_{\mu_p}\left(F(\eta_s), F(\eta_{s-r})\right)
={\rm Cov}_{\nu_p}\left(F(\eta_r), F(\eta_{0})\right).
\end{equation}
By Lemma \ref{lemma 3.zero}, we have
\[
|{\rm Cov}_{\mu_p}\left(F(\eta_s), F(\eta_{s-r})\right)|
\leq \frac{2\|H\|_\infty^2}{C_2}\sum_{i=1}^m\sum_{j=1}^m\int_0^{+\infty}p_{v+r}(x_i, x_j)dv
\]
for any $0\leq r\leq s$. Then, by Assumption \eqref{Assumption 1.6} and the dominated convergence theorem,
\[
\lim_{s\rightarrow+\infty}\int_0^s{\rm Cov}_{\mu_p}\left(F(\eta_s), F(\eta_{s-r})\right)dr
=\int_0^{+\infty}{\rm Cov}_{\nu_p}\left(F(\eta_r), F(\eta_{0})\right)dr=\frac{\sigma_F^2}{2}
\]
and hence Lemma \ref{lemma 3.3} holds.

\qed

Now we utilize the martingale decomposition method to prove Theorem \ref{theorem 2.3 main}. We give details of our proof in the case where $\eta_0$ is distributed with $\nu_p$. For the $\mu_p$ case, since the argument is similar, we only give an outline of the proof.

As we have introduced at the end of Section \ref{section two}, we denote by $\{S(t)\}_{t\geq 0}$ the Markov semigroup of $\{\eta_t\}_{t\geq 0}$, i..e,
\[
S(t)F(\eta)=\mathbb{E}_\eta F(\eta_t)
\]
for any $t\geq 0$. According to Proposition \ref{proposition 3.2 dual equivalent},
\[
\mathbb{E}_\eta F(\eta_t)=\sum_{y_1, \ldots, y_m}\beta_t\left((x_1, \ldots, x_m), (y_1, \ldots, y_m)\right)H(\eta, y_1, \ldots, y_m).
\]
By Kolmogorov backward equation,
\begin{align}\label{equ 3.11}
S(N)F(\eta)-F(\eta)&=S(N)F(\eta)-S(0)F(\eta) \notag\\
&=\int_0^N\frac{d}{ds}S(s)F(\eta)ds=\int_0^N\mathcal{L}S(s)F(\eta)ds\notag\\
&=\mathcal{L}\left(\int_0^NS(s)Fds\right)(\eta).
\end{align}
For each $N\geq 1$ and $\eta\in \{0, 1\}^G$, we define
\begin{align*}
G^N(\eta)&=\left(\int_0^NS(s)Fds\right)(\eta)\\
&=\sum_{y_1, \ldots, y_m\in G}H(\eta, y_1, \ldots, y_m)\int_0^N\beta_s\left((x_1, \ldots, x_m), (y_1, \ldots, y_m)\right)ds.
\end{align*}

We further define
\[
M_t^N=G^N(\eta_t)-G^N(\eta_0)-\int_0^t\mathcal{L}G^N(\eta_s)ds.
\]
By Dynkin's martingale formula, $\{M_t^N\}_{t\geq 0}$ is a martingale. By \eqref{equ 3.11},
\begin{equation}\label{equ 3.11 two}
M_t^N=G^N(\eta_t)-G^N(\eta_0)-\int_0^tS(N)F(\eta_s)ds+\int_0^t F(\eta_s)ds.
\end{equation}
Without loss of generality, from now on we further assume that $\mathbb{E}_{\nu_p}F(\eta_0)=0$ and then $\xi_t^F=\int_0^t F(\eta_s)ds$ when $\eta_0$ is distributed with $\nu_p$. Therefore, when $\eta_0$ is distributed with $\nu_p$,
\begin{equation}\label{equ 3.12}
\frac{1}{\sqrt{N}}M_{tN}^N=\frac{1}{\sqrt{N}}G^N(\eta_{tN})-\frac{1}{\sqrt{N}}G^N(\eta_0)-\frac{1}{\sqrt{N}}
\int_0^{tN}S(N)F(\eta_s)ds+\frac{1}{\sqrt{N}}\xi_{tN}^F.
\end{equation}

The following lemma shows that the term
\[
\frac{1}{\sqrt{N}}G^N(\eta_{tN})-\frac{1}{\sqrt{N}}G^N(\eta_0)-\frac{1}{\sqrt{N}}
\int_0^{tN}S(N)F(\eta_s)ds
\]
in decomposition \eqref{equ 3.12} is a small error.

\begin{lemma}\label{lemma 3.4}
Let $\eta_0$ be distributed with $\nu_p$.
For any $t\geq 0$,
\[
\lim_{N\rightarrow+\infty}\left(\frac{1}{\sqrt{N}}G^N(\eta_{tN})-\frac{1}{\sqrt{N}}G^N(\eta_0)-\frac{1}{\sqrt{N}}
\int_0^{tN}S(N)F(\eta_s)ds\right)=0
\]
in $L^2$.
\end{lemma}

\proof
According to the invariance of $\nu_p$, we only need to show that
\begin{equation}\label{equ 3.13}
\lim_{N\rightarrow+\infty}{\rm Var}_{\nu_p}\left(\frac{1}{\sqrt{N}}G^N(\eta_0)\right)=0
\end{equation}
and
\begin{equation}\label{equ 3.14}
\lim_{N\rightarrow+\infty}{\rm Var}_{\nu_p}\left(\frac{1}{\sqrt{N}}
\int_0^{tN}S(N)F(\eta_s)ds\right)=0.
\end{equation}
By the bilinear property of the covariance,
\begin{align*}
{\rm Var}_{\nu_p}\left(G^N(\eta_0)\right)
&=\sum_{y_1, \ldots, y_m, z_1, \ldots, z_m\in G}\Bigg({\rm Cov}_{\nu_p}\left(H(\eta_0, y_1, \ldots, y_m), H(\eta_0, z_1, \ldots, z_m)\right)\\
&\text{\quad\quad}\times\int_0^N\int_0^N\alpha_{\vec{x}}\left(\vec{y}, \vec{z}, s_1, s_2\right)ds_1ds_2\Bigg),
\end{align*}
where $\vec{x}=(x_1, \ldots, x_m), \vec{y}=(y_1, \ldots, y_m), \vec{z}=(z_1, \ldots, z_m)$ and
\[
\alpha_{x_1, \ldots, x_m}\left(\vec{y}, \vec{z}, s_1, s_2\right)
=\beta_{s_1}\left(\vec{x}, \vec{y}\right)\beta_{s_2}\left(\vec{x}, \vec{z}\right).
\]
By Lemma \ref{lemma 3.zero}, we have
\begin{align*}
\left|{\rm Cov}_{\nu_p}\left(H(\eta_0, \vec{y}), H(\eta_0, \vec{z})\right)\right|
&=\lim_{t\rightarrow+\infty}\left|{\rm Cov}_{\mu_p}\left(H(\eta_t, \vec{y}), H(\eta_t, \vec{z})\right)\right|\\
&\leq \frac{2\|H\|_\infty^2}{C_2}\sum_{i=1}^m\sum_{j=1}^m\int_0^{+\infty}p_s(y_i, z_j)ds.
\end{align*}
According to the symmetry of $p_t$ and the Markov property of the coalescing random walk,
\[
\sum_{y_1, \ldots, y_m, z_1, \ldots, z_m\in G}\beta_{s_1}\left(\vec{x}, \vec{y}\right)\beta_{s_2}\left(\vec{x}, \vec{z}\right)
p_s(y_i, z_j)=p_{s+s_1+s_2}(x_i, x_j).
\]
As a result,
\begin{align*}
\frac{1}{N}{\rm Var}_{\nu_p}\left(G^N(\eta_0)\right)&=\frac{O(1)}{N}\sum_{i=1}^m\sum_{j=1}^m\int_0^{+\infty}\left(\int_0^N\int_0^Np_{s+s_1+s_2}(x_i, x_j)d{s_1}ds_2\right)ds\\
&\leq \frac{O(1)}{N}\sum_{i=1}^m\sum_{j=1}^m\int_0^N\left(\int_0^{+\infty}\int_0^{+\infty}p_{s+s_1+s_2}(x_i, x_j)dsds_2\right)ds_1.
\end{align*}
Hence, to prove \eqref{equ 3.13}, we only need to show that
\begin{equation}\label{equ 3.15}
\lim_{s_1\rightarrow+\infty}\int_0^{+\infty}\int_0^{+\infty}p_{s+s_1+s_2}(x_i, x_j)dsds_2=0.
\end{equation}
Since
\[
\int_0^{+\infty}\int_0^{+\infty}p_{s+s_1+s_2}(x_i, x_j)dsds_2=\int^{+\infty}_{\frac{s_1}{2}}\int^{+\infty}_{\frac{s_1}{2}}
p_{s+s_2}(x_i, x_j)dsds_2,
\]
Equation \eqref{equ 3.15} follows from Assumption \eqref{Assumption 1.6} and hence \eqref{equ 3.13} holds.

Now we verify \eqref{equ 3.14}. By the bilinear property of the covariance, the invariance of $\nu_p$ and the definition of $S(N)$,
\begin{align*}
&{\rm Var}_{\nu_p}\left(\int_0^{tN}S(N)F(\eta_s)ds\right)\\
&=2\int_0^{tN}ds_1\int_0^{s_1}\sum_{\vec{y}, \vec{z}\in G^m}\beta_N(\vec{x}, \vec{y})\beta_N(\vec{x}, \vec{z})
{\rm Cov}_{\nu_p}\left(H(\eta_{s_2}, \vec{y}), H(\eta_0, \vec{z})\right)ds_2.
\end{align*}
By Lemma \ref{lemma 3.zero},
\begin{align*}
|{\rm Cov}_{\nu_p}\left(H(\eta_{s_2}, \vec{y}), H(\eta_0, \vec{z})\right)|
&=\lim_{t\rightarrow+\infty}|{\rm Cov}_{\mu_p}\left(H(\eta_{s_2+t}, \vec{y}), H(\eta_t, \vec{z})\right)|\\
&\leq \frac{2\|H\|_\infty^2}{C_2}\sum_{i=1}^m\sum_{j=1}^m\int_0^{+\infty}p_{v+s_2}(y_i, z_j)dv.
\end{align*}
According to the symmetry of $p_t$ and the Markov property of the coalescing random walk,
\[
\sum_{\vec{y}, \vec{z}\in G^m}\beta_{N}\left(\vec{x}, \vec{y}\right)\beta_{N}\left(\vec{x}, \vec{z}\right)
p_{v+s_2}(y_i, z_j)=p_{2N+v+s_2}(x_i, x_j).
\]
As a result,
\[
{\rm Var}_{\nu_p}\left(\int_0^{tN}S(N)F(\eta_s)ds\right)
\leq O(1)N\sum_{1\leq i, j\leq m}\int_0^{+\infty}\int_0^{+\infty}p_{2N+v+s_2}(x_i, x_j)dv ds_2.
\]
Therefore, to prove \eqref{equ 3.14}, we only need to show that
\begin{equation}\label{equ 3.16}
\lim_{N\rightarrow+\infty}\int_0^{+\infty}\int_0^{+\infty}p_{2N+v+s_2}(x_i, x_j)dv ds_2=0.
\end{equation}
Since
\[
\int_0^{+\infty}\int_0^{+\infty}p_{2N+v+s_2}(x_i, x_j)dv ds_2=\int_N^{+\infty}\int_N^{+\infty}p_{v+s_2}(x_i, x_j)dv ds_2,
\]
Equation \eqref{equ 3.16} follows from Assumption \eqref{Assumption 1.6} and hence \eqref{equ 3.14} holds. Since \eqref{equ 3.13} and \eqref{equ 3.14} hold, the proof is complete.
\qed

The following lemma gives the invariance principle of the martingale term in decomposition \eqref{equ 3.12}.

\begin{lemma}\label{lemma 3.5}
Let $T>0$ and $\eta_0$ be distributed with $\nu_p$. As $N\rightarrow+\infty$, the process
\[
\left\{\frac{1}{\sqrt{N}}M_{tN}^N:~0\leq t\leq T\right\}
\]
converges weakly, with respect to the Skorohod topology of $D[0, T]$, to $\{\sigma_FB_t\}_{0\leq t\leq T}$.
\end{lemma}

\proof
We denote by $\{\langle M^N\rangle_t\}_{t\geq 0}$ the predictable quadratic variation process of $\{M_t^N\}_{t\geq 0}$.  By Theorem 1.4 in Chapter 7 of \cite{Ethier1986}, to complete the proof, we only need to show that
\[
\lim_{N\rightarrow+\infty}\frac{1}{N}\langle M^N\rangle_{tN}=\sigma_F^2t
\]
in probability for any $t\geq 0$. Hence, we only need to show that
\begin{equation}\label{equ 3.17}
\lim_{N\rightarrow+\infty}\frac{1}{N}\mathbb{E}_{\nu_p}\langle M^N\rangle_{tN}=\sigma_F^2 t
\end{equation}
and
\begin{equation}\label{equ 3.18}
\lim_{N\rightarrow+\infty}{\rm Var}_{\nu_p}\left(\frac{1}{N}\langle M^N\rangle_{tN}\right)=0.
\end{equation}
Since $\mathbb{E}_{\nu_p}\langle M^N\rangle_{t}={\rm Var}_{\nu_p}(M_t^N)$, Equation \eqref{equ 3.17} follows from \eqref{equ 3.12} and Lemma \ref{lemma 3.4}. Hence we only need to verify \eqref{equ 3.18}. By Dynkin's martingale formula and definition of $\mathcal{L}$,
\[
\langle M^N\rangle_t=\int_0^t\left(\sum_{w\in G}\sum_{v\in G, v\neq w}q(w,v)\left(G^N(\eta_s^{w,v})-G^N(\eta_s)\right)^2\right)ds.
\]
According to the definition of $G^N$,
\[
G^N(\eta^{w, v})-G^N(\eta)=\sum_{\vec{y}\in G^m}\left(H(\eta^{w, v}, \vec{y})-H(\eta, \vec{y})\right)\int_0^N\beta_s(\vec{x}, \vec{y})ds.
\]
When $y_i\neq w$ for all $1\leq i\leq m$, $H(\eta^{w, v}, \vec{y})=H(\eta, \vec{y})$. Hence,
\[
G^N(\eta^{w, v})-G^N(\eta)=\sum_{\vec{y}\in \mathcal{B}(w)}\left(H(\eta^{w, v}, \vec{y})-H(\eta, \vec{y})\right)\int_0^N\beta_s(\vec{x}, \vec{y})ds,
\]
where
\[
\mathcal{B}(w)=\left\{\vec{y}\in G^m:~y_i=w\text{~for some~}1\leq i\leq m\right\}.
\]
Consequently,
\begin{align}\label{equ 3.19}
&\left(G^N(\eta^{w, v})-G^N(\eta)\right)^2\\
&=\sum_{\vec{y}, \vec{z}\in \mathcal{B}(w)}\left(H(\eta^{w, v}, \vec{y})-H(\eta, \vec{y})\right)\left(H(\eta^{w, v}, \vec{z})-H(\eta, \vec{z})\right)\int_0^N\int_0^N\beta_{s_1}(\vec{x}, \vec{y})\beta_{s_2}(\vec{x}, \vec{z})ds_1ds_2.  \notag
\end{align}
According to the definition of $\mathcal{B}(w)$,
\begin{align*}
\sum_{\vec{y}\in \mathcal{B}(w)}\beta_s(\vec{x}, \vec{y})
=\hat{\mathbb{P}}\left(\mathcal{X}_s^{x_i}=w\text{~for some~}i\right)\leq \sum_{i=1}^mp_s(x_i, w).
\end{align*}
Hence,
\begin{align*}
&\sum_{w\in G}\sum_{v\in G, v\neq w}q(w, v)\sum_{\vec{y}, \vec{z}\in \mathcal{B}(w)}\int_0^N\int_0^N\beta_{s_1}(\vec{x}, \vec{y})\beta_{s_2}(\vec{x}, \vec{y})ds_1ds_2\\
&\leq C_5\sum_{i=1}^m\sum_{j=1}^m\int_0^{+\infty}\int_0^{+\infty}\sum_{w\in G}p_{s_1}(x_i, w)p_{s_2}(x_j, w)ds_1ds_2\\
&= C_5\sum_{i=1}^m\sum_{j=1}^m\int_0^{+\infty}\int_0^{+\infty}p_{s_1+s_2}(x_i, x_j)ds_1ds_2<+\infty
\end{align*}
under Assumption \eqref{Assumption 1.6}, where $C_5=\sup_{w\in G}\sum_{v\in G, v\neq w}q(w, v)<+\infty$. Then, by \eqref{equ 3.19} and the fact that
\[
\left|\left(H(\eta^{w, v}, \vec{y})-H(\eta, \vec{y})\right)\left(H(\eta^{w, v}, \vec{z})-H(\eta, \vec{z})\right)\right|
\leq 4\|H\|_\infty^2,
\]
to verify \eqref{equ 3.18} we only need to show that
\begin{equation}\label{equ 3.20}
\lim_{s\rightarrow+\infty}\left(\sup_{r\geq 0}\left|{\rm Cov}_{\nu_p}\left(\tilde{H}(\eta_{r+s}, \vec{y}, \vec{z}, w, v), \tilde{H}(\eta_{r}, \vec{u}, \vec{h}, \tilde{w}, \tilde{v})\right)\right|\right)=0
\end{equation}
for any $\vec{y}, \vec{z}, \vec{u}, \vec{h}\in G^m, w, v, \tilde{w}, \tilde{v}\in G$ according to the bilinear property of the covariance and the dominated convergence theorem, where
\[
\tilde{H}(\eta, \vec{y}, \vec{z}, w, v)=\left(H(\eta^{w, v}, \vec{y})-H(\eta, \vec{y})\right)\left(H(\eta^{w, v}, \vec{z})-H(\eta, \vec{z})\right).
\]
By Lemma \ref{lemma 3.zero},
\begin{align*}
&\left|{\rm Cov}_{\nu_p}\left(\tilde{H}(\eta_{r+s}, \vec{y}, \vec{z}, w, v), \tilde{H}(\eta_{r}, \vec{u}, \vec{h}, \tilde{w}, \tilde{v})\right)\right|\\
&=\lim_{t\rightarrow+\infty}\left|{\rm Cov}_{\mu_p}\left(\tilde{H}(\eta_{r+s+t}, \vec{y}, \vec{z}, w, v), \tilde{H}(\eta_{r+t}, \vec{u}, \vec{h}, \tilde{w}, \tilde{v})\right)\right|\\
&\leq \frac{2\|\tilde{H}\|_\infty^2}{C_2}\sum_{i=1}^{2m+2}\sum_{j=1}^{2m+2}\int_0^{+\infty}p_{\theta+s}\left(\varsigma_i, \varpi_j\right)d\theta,
\end{align*}
where
\[
\vec{\varsigma}=\left(\vec{y}, \vec{z}, w, v\right) \text{~and~} \vec{\varpi}=\left(\vec{u}, \vec{h}, \tilde{w}, \tilde{v}\right).
\]
Consequently, Equation \eqref{equ 3.20} follows from Assumption \eqref{Assumption 1.6}. Since \eqref{equ 3.20} holds, the proof is complete.

\qed

By Equation \eqref{equ 3.12} and Lemmas \ref{lemma 3.4}, \ref{lemma 3.5}, we have the following corollary.

\begin{corollary}\label{corollary 3.6}
Let $\eta_0$ be distributed with $\nu_p$. For any $0<t_1<t_2<\ldots<t_k$, the $\mathbb{R}^k$-valued random variable
\[
\left(\frac{1}{\sqrt{N}}\xi_{t_1N}^F, \frac{1}{\sqrt{N}}\xi_{t_2N}^F, \ldots, \frac{1}{\sqrt{N}}\xi_{t_kN}^F\right)
\]
converges weakly to $\sigma_F\left(B_{t_1}, B_{t_2}, \ldots, B_{t_k}\right)$ as $N\rightarrow+\infty$.
\end{corollary}

By Corollary \ref{corollary 3.6}, we only need to verify that $\left\{\frac{1}{\sqrt{N}}\xi_{tN}^F:~0\leq t\leq T\right\}_{N\geq 1}$ are tight with respect to the uniform topology of $C[0, T]$ to complete the proof of Theorem \ref{theorem 2.3 main}. We will show that
\[
\mathbb{E}_{\nu_p}\left(\left(\xi_{tN}^F-\xi_{sN}^F\right)^4\right)=O(1)(t-s)^2N^2
\]
for $s<t$ and then the aforesaid tightness follows from Corollary 14.9 of \cite{Kallenberg1997}. According to Proposition \ref{proposition 3.2 dual equivalent} and the definitions of $\xi_t^F, \nu_p$,
\[
\mathbb{E}_{\nu_p}\left(\left(\xi_{t}^F-\xi_{s}^F\right)^4\right)=24\int_{s}^tds_1\int_s^{s_1}ds_2\int_s^{s_2}ds_3\int_s^{s_4}
\mathbb{E}_{\nu_p}\left(\prod_{j=1}^4F(\eta_{s_j})\right)ds_3
\]
and
\begin{align*}
&\mathbb{E}_{\nu_p}\left(\prod_{j=1}^4F(\eta_{s_j})\right)
=\lim_{\theta\rightarrow+\infty}\mathbb{E}_{\mu_p}\left(\prod_{j=1}^4\left(F(\eta_{\theta+s_j})-\mathbb{E}_{\mu_p}F(\eta_{\theta+s_j})\right)\right)\\
&=\lim_{\theta\rightarrow+\infty}\hat{\mathbb{E}}\mathbb{E}_{\mu_p}\left(\prod_{j=1}^4\left(H(\eta_0, \mathcal{X}_{\theta+s_1}^{s_1-s_j, x_1}, \ldots,
\mathcal{X}_{\theta+s_1}^{s_1-s_j, x_m})-\mathbb{E}_{\mu_p}F(\eta_{\theta+s_j})\right)\right)
\end{align*}
for $F(\eta)=H(\eta, x_1,\ldots, x_m)$,
so we need to bound
\[
\hat{\mathbb{E}}\mathbb{E}_{\mu_p}\left(\prod_{j=1}^4\left(H(\eta_0, \mathcal{X}_{s_1}^{s_1-s_j, x_1}, \ldots,
\mathcal{X}_{s_1}^{s_1-s_j, x_m})-\mathbb{E}_{\mu_p}F(\eta_{s_j})\right)\right)
\]
from above for any $s_4\leq s_3\leq s_2\leq s_1$. For this purpose, we introduce some notations. For $0\leq s_4\leq s_3\leq s_2\leq s_1$ and $1\leq k\leq 4$, we define
\[
\left\{\left(\mathcal{X}_{t, k}^{s_1-s_k, x_1}, \mathcal{X}_{t, k}^{s_1-s_k, x_2}, \ldots, \mathcal{X}_{t, k}^{s_1-s_k, x_m}\right)\right\}_{t\geq 0}
\]
as a copy the coalescing random walk $\left\{\left(\mathcal{X}_{t}^{s_1-s_k, x_1}, \mathcal{X}_{t}^{s_1-s_k, x_2}, \ldots, \mathcal{X}_{t}^{s_1-s_k, x_m}\right)\right\}_{t\geq 0}$. We further assume that these four different coalescing random walks are independent. Now we couple these four independent coalescing random walks and the coalescing random walk
\begin{align*}
&\Bigg\{\Big(\mathcal{X}_{t}^{s_1-s_1, x_1}, \mathcal{X}_{t}^{s_1-s_1, x_2}, \ldots, \mathcal{X}_{t}^{s_1-s_1, x_m},
\mathcal{X}_{t}^{s_1-s_2, x_1} \ldots, \mathcal{X}_{t}^{s_1-s_2, x_m}, \\
&\text{\quad\quad}\mathcal{X}_{t}^{s_1-s_3, x_1}, \ldots, \mathcal{X}_{t}^{s_1-s_3, x_m}
, \mathcal{X}_{t}^{s_1-s_4, x_1}, \ldots, \mathcal{X}_{t}^{s_1-s_4, x_m}\Big)\Bigg\}_{t\geq 0}
\end{align*}
in the same probability space. For $0\leq s_4\leq s_3\leq s_2\leq s_1$ and $1\leq l<k\leq 4$, we define
\begin{equation}\label{equ 3.21}
\tau_{l, k}^{s_1, s_2, s_3, s_4}
=\inf\left\{t\geq s_1-s_k:~\mathcal{X}_{t, k}^{s_1-s_k, x_i}=\mathcal{X}_{t, l}^{s_1-s_l, x_j}\text{~for some~}1\leq i\leq j\leq m\right\}.
\end{equation}
We further define $\tilde{\tau}^{s_1, s_2, s_3, s_4}=\inf\{\tau_{l, k}^{s_1, s_2, s_3, s_4}:~1\leq l<k\leq 4\}$. For $t<\tilde{\tau}^{s_1, s_2, s_3, s_4}$, there is no collision for any two active coordinates respectively from two independent coalescing random walks and hence we can define
\[
\mathcal{X}_{t}^{s_1-s_k, x_j}=\mathcal{X}_{t, k}^{s_1-s_k, x_j}
\]
for any $1\leq k\leq 4$, $1\leq j\leq m$ and $t<\tilde{\tau}^{s_1, s_2, s_3, s_4}$ in the sense of coupling. For $t\geq \tilde{\tau}^{s_1, s_2, s_3, s_4}$, since the coalescent of two independent active walks can be equivalently considered as one walk is erased and another is remained, we define $\mathcal{X}_{t}^{s_1-s_k, x_j}$ as follows. When $\tilde{\tau}^{s_1, s_2, s_3, s_4}=\tau_{l, k}^{s_1, s_2, s_3, s_4}$ for $1\leq l<k\leq 4$, we denote by $\hat{l}, \hat{k}$ the two elements in $\{1, 2, 3, 4\}\setminus \{l, k\}$ such that $\hat{l}<\hat{k}$. At any moment $t\geq \tilde{\tau}^{s_1, s_2, s_3, s_4}$, when an active coordinate from $\left(\mathcal{X}_{\cdot, k}^{s_1-s_k, x_1}, \mathcal{X}_{\cdot, k}^{s_1-s_k, x_2}, \ldots, \mathcal{X}_{\cdot, k}^{s_1-s_k, x_m}\right)$ collides with another active coordinate from $\left(\mathcal{X}_{\cdot, l}^{s_1-s_l, x_1}, \mathcal{X}_{\cdot, l}^{s_1-s_l, x_2}, \ldots, \mathcal{X}_{\cdot, l}^{s_1-s_l, x_m}\right)$, then the one from
\[
\left(\mathcal{X}_{\cdot, l}^{s_1-s_l, x_1}, \mathcal{X}_{\cdot, l}^{s_1-s_l, x_2}, \ldots, \mathcal{X}_{\cdot, l}^{s_1-s_l, x_m}\right)
\]
is remained and the one from $\left(\mathcal{X}_{\cdot, k}^{s_1-s_k, x_1}, \mathcal{X}_{\cdot, k}^{s_1-s_k, x_2}, \ldots, \mathcal{X}_{\cdot, k}^{s_1-s_k, x_m}\right)$ is erased. Similarly, when an active coordinate from $\left(\mathcal{X}_{\cdot, \hat{k}}^{s_1-s_{\hat{k}}, x_1}, \mathcal{X}_{\cdot, \hat{k}}^{s_1-s_{\hat{k}}, x_2}, \ldots, \mathcal{X}_{\cdot, \hat{k}}^{s_1-s_{\hat{k}}, x_m}\right)$ collides with another active coordinate from $\left(\mathcal{X}_{\cdot, \hat{l}}^{s_1-s_{\hat{l}}, x_1}, \mathcal{X}_{\cdot, \hat{l}}^{s_1-s_{\hat{l}}, x_2}, \ldots, \mathcal{X}_{\cdot, \hat{l}}^{s_1-s_{\hat{l}}, x_m}\right)$, then the one from
\[
\left(\mathcal{X}_{\cdot, \hat{l}}^{s_1-s_{\hat{l}}, x_1}, \mathcal{X}_{\cdot, \hat{l}}^{s_1-s_{\hat{l}}, x_2}, \ldots, \mathcal{X}_{\cdot, \hat{l}}^{s_1-s_{\hat{l}}, x_m}\right)
\]
is remained and the one from $\left(\mathcal{X}_{\cdot, \hat{k}}^{s_1-s_{\hat{k}}, x_1}, \mathcal{X}_{\cdot, \hat{k}}^{s_1-s_{\hat{k}}, x_2}, \ldots, \mathcal{X}_{\cdot, \hat{k}}^{s_1-s_{\hat{k}}, x_m}\right)$ is erased. Furthermore, when an active coordinate from
\[
\left(\mathcal{X}_{\cdot, k}^{s_1-s_k, x_1}, \mathcal{X}_{\cdot, k}^{s_1-s_k, x_2}, \ldots, \mathcal{X}_{\cdot, k}^{s_1-s_k, x_m},
\mathcal{X}_{\cdot, l}^{s_1-s_l, x_1}, \mathcal{X}_{\cdot, l}^{s_1-s_l, x_2}, \ldots, \mathcal{X}_{\cdot, l}^{s_1-s_l, x_m}\right)
\]
collides with another active coordinate from
\[
\left(\mathcal{X}_{\cdot, \hat{k}}^{s_1-s_{\hat{k}}, x_1}, \mathcal{X}_{\cdot, \hat{k}}^{s_1-s_{\hat{k}}, x_2}, \ldots, \mathcal{X}_{\cdot, \hat{k}}^{s_1-s_{\hat{k}}, x_m}, \mathcal{X}_{\cdot, \hat{l}}^{s_1-s_{\hat{l}}, x_1}, \mathcal{X}_{\cdot, \hat{l}}^{s_1-s_{\hat{l}}, x_2}, \ldots, \mathcal{X}_{\cdot, \hat{l}}^{s_1-s_{\hat{l}}, x_m}\right),
\]
then the one from
\[
\left(\mathcal{X}_{\cdot, \hat{k}}^{s_1-s_{\hat{k}}, x_1}, \mathcal{X}_{\cdot, \hat{k}}^{s_1-s_{\hat{k}}, x_2}, \ldots, \mathcal{X}_{\cdot, \hat{k}}^{s_1-s_{\hat{k}}, x_m}, \mathcal{X}_{\cdot, \hat{l}}^{s_1-s_{\hat{l}}, x_1}, \mathcal{X}_{\cdot, \hat{l}}^{s_1-s_{\hat{l}}, x_2}, \ldots, \mathcal{X}_{\cdot, \hat{l}}^{s_1-s_{\hat{l}}, x_m}\right)
\]
is remained and the one from
\[
\left(\mathcal{X}_{\cdot, k}^{s_1-s_k, x_1}, \mathcal{X}_{\cdot, k}^{s_1-s_k, x_2}, \ldots, \mathcal{X}_{\cdot, k}^{s_1-s_k, x_m},
\mathcal{X}_{\cdot, l}^{s_1-s_l, x_1}, \mathcal{X}_{\cdot, l}^{s_1-s_l, x_2}, \ldots, \mathcal{X}_{\cdot, l}^{s_1-s_l, x_m}\right)
\]
is erased.

For any different $l, k\in \{1, 2, 3, 4\}$ and $x, y\in \{x_1, x_2, \ldots, x_m\}$, we denote by $A_{l, k, x, y}^{s_1, s_2, s_3, s_4}$ the event that
\[
\mathcal{X}_{t,l}^{s_1-s_l, x}=\mathcal{X}_{t, k}^{s_1-s_k, y}
\]
for some $t\geq \max\{s_1-s_k, s_1-s_l\}$. We have the following lemma.

\begin{lemma}\label{lemma 3.7}
For any $1\leq l<k\leq 4$ and $x, y\in \{x_1, x_2, \ldots, x_m\}$,
\[
\int_0^{s_l}\hat{\mathbb{P}}\left(A_{l, k, x, y}^{s_1, s_2, s_3, s_4}\right)ds_k\leq \frac{C_3}{C_2},
\]
where $C_2$ is defined as in Assumption \eqref{Assumption 1.5} and $C_3$ is defined as in Assumption \eqref{Assumption 1.6}.
\end{lemma}

\proof
According to the Markov property,
\begin{align*}
\hat{\mathbb{P}}\left(A_{l, k, x, y}^{s_1, s_2, s_3, s_4}\right)
=\sum_{z\in G}p_{s_l-s_k}(x, z)\hat{\mathbb{P}}\left(\mathcal{X}_{t, l}^{0, z}=\mathcal{X}_{t, k}^{0, y}\text{~for some~}t>0\right).
\end{align*}
As we have shown in the proof of Lemma \ref{lemma asymptotic variance},
\[
\hat{\mathbb{P}}\left(\mathcal{X}_{t, l}^{0, z}=\mathcal{X}_{t, k}^{0, y}\text{~for some~}t>0\right)\leq \frac{1}{C_2}\int_0^{+\infty}p_s(z, y)ds
\]
and hence
\[
\hat{\mathbb{P}}(A_{l, k, x, y}^{s_1, s_2, s_3, s_4})
\leq \frac{1}{C_2}\int_0^{+\infty}p_{s_l-s_k+\theta}(x, y)d\theta.
\]
Therefore, Lemma \ref{lemma 3.7} follows from Assumption \eqref{Assumption 1.6}.
\qed

For different $l, k, j\in \{1, 2, 3, 4\}$ and $x, y, z\in \{x_1, \ldots, x_m\}$, we denote by $\mathcal{H}_{l, k, j, x, y, z}^{s_1, s_2, s_3, s_4}$ the event
\[
A_{l, k, x, y}^{s_1, s_2, s_3, s_4}\bigcap A_{l, j, x, z}^{s_1, s_2, s_3, s_4}.
\]
We have the following lemma.

\begin{lemma}\label{lemma 3.8}
For different $l, k, j\in \{1, 2, 3, 4\}$ and $x, y, z\in \{x_1, \ldots, x_m\}$,
\[
\int_0^{s_a}ds_b\int_0^{s_b}\hat{\mathbb{P}}\left(\mathcal{H}_{l, k, j, x, y, z}^{s_1, s_2, s_3, s_4}\right)ds_c\leq 3\left(C_3/C_2\right)^2,
\]
where $a=\min\{l, k, j\}, c=\max\{l, k, j\}$ and $b=\{l, k, j\}\setminus\{a, c\}$.
\end{lemma}

\proof We only deal with the case $(l, k, j)=(1, 2, 3)$, since other cases follow from similar arguments.
Let
\[
\tau_1=\min\{t\geq s_1-s_2:~\mathcal{X}_{t, 1}^{0, x}=\mathcal{X}_{t, 2}^{s_1-s_2, y}\}
\]
and
\[
\tau_2=\min\{t\geq s_1-s_3:~\mathcal{X}_{t, 1}^{0, x}=\mathcal{X}_{t, 3}^{s_1-s_3, z}\},
\]
then
\begin{align*}
\hat{\mathbb{P}}\left(\mathcal{H}_{l, k, j, x, y, z}^{s_1, s_2, s_3, s_4}\right)
=&\hat{\mathbb{P}}\left(\mathcal{H}_{l, k, j, x, y, z}^{s_1, s_2, s_3, s_4}, \tau_1\leq s_1-s_3\right)
+\hat{\mathbb{P}}\left(\mathcal{H}_{l, k, j, x, y, z}^{s_1, s_2, s_3, s_4}, s_1-s_3\leq \tau_1<\tau_2\right)\\
&+\hat{\mathbb{P}}\left(\mathcal{H}_{l, k, j, x, y, z}^{s_1, s_2, s_3, s_4}, \tau_2<\tau_1\right).
\end{align*}
According to the strong Markov property,
\begin{align*}
&\hat{\mathbb{P}}\left(\mathcal{H}_{l, k, j, x, y, z}^{s_1, s_2, s_3, s_4}, s_1-s_3\leq \tau_1<\tau_2\right)\\
&\leq \sum_{w_1\in G}\sum_{w_2\in G}\int_{s_2-s_3}^{+\infty}1_{\{v=w_1\}}p_{u-(s_2-s_3)}(z, w_2)\hat{\mathbb{P}}\left(\mathcal{X}_{t, 1}^{0, w_1}=\mathcal{X}_{t, 2}^{0, w_2}\text{~for some~}t\geq 0\right)m(du, dv),
\end{align*}
where $m(du, dv)$ is the joint distribution of $\left(\tau_1-(s_1-s_2), X_{\tau_1, 1}^{0, x}\right)$. As we have shown in the proof of Lemma \ref{lemma asymptotic variance},
\[
\hat{\mathbb{P}}\left(\mathcal{X}_{t, 1}^{0, w_1}=\mathcal{X}_{t, 2}^{0, w_2}\text{~for some~}t\geq 0\right)
\leq \frac{1}{C_2}\int_0^{+\infty}p_{\theta}(w_1, w_2)d\theta
\]
and hence
\begin{align*}
&\hat{\mathbb{P}}\left(\mathcal{H}_{l, k, j, x, y, z}^{s_1, s_2, s_3, s_4}, s_1-s_3\leq \tau_1<\tau_2\right)\\
&\leq \frac{1}{C_2}\sum_{w_1\in G}\int_{s_2-s_3}^{+\infty}1_{\{v=w_1\}}\left(\int_0^{+\infty}p_{u-(s_2-s_3)+\theta}(z, w_1) d\theta\right) m(du, dv).
\end{align*}
Consequently, by Assumption \eqref{Assumption 1.6},
\begin{align*}
&\int_0^{s_2}\hat{\mathbb{P}}\left(\mathcal{H}_{1, 2, 3, x, y, z}^{s_1, s_2, s_3, s_4}, s_1-s_3\leq \tau_1<\tau_2\right)ds_3\\
&\leq \frac{\sum_{w_1\in G}}{C_2}\int_0^{s_2}\left(\int_{s_2-s_3}^{+\infty}1_{\{v=w_1\}}\left(\int_0^{+\infty}p_{u-(s_2-s_3)+\theta}(z, w_1) d\theta\right) m(du, dv)\right)ds_3\\
&=\frac{\sum_{w_1\in G}}{C_2}\int_0^{+\infty}1_{\{v=w_1\}}\left(\int^{s_2}_{\max\{0, s_2-u\}}\left(\int_0^{+\infty}p_{u-(s_2-s_3)+\theta}(z, w_1) d\theta\right)ds_3\right)m(du, dv)\\
&\leq \frac{C_3}{C_2}\sum_{w_1\in G}\int_0^{+\infty}1_{\{v=w_1\}}m(du, dv)=\frac{C_3}{C_2}\int_0^{+\infty}1m(du)=\frac{C_3}{C_2}\hat{\mathbb{P}}\left(\tau_1<+\infty\right)\\
&=\frac{C_3}{C_2}\hat{\mathbb{P}}\left(A_{1, 2, x, y}^{s_1, s_2, s_3, s_4}\right).
\end{align*}
Hence, by Lemma \ref{lemma 3.7},
\[
\int_0^{s_1}ds_2\int_0^{s_2}\hat{\mathbb{P}}\left(\mathcal{H}_{1, 2, 3, x, y, z}^{s_1, s_2, s_3, s_4}, s_1-s_3\leq \tau_1<\tau_2\right)ds_3
\leq \left(C_3/C_2\right)^2.
\]
According to similar arguments, the above inequality still holds when $\{s_1-s_3\leq \tau_1<\tau_2\}$ is replaced by $\{\tau_1\leq s_1-s_3\}$ or $\{\tau_2<\tau_1\}$ and the proof is complete.
\qed

To bound
\[
\hat{\mathbb{E}}\mathbb{E}_{\mu_p}\left(\prod_{j=1}^4\left(H(\eta_0, \mathcal{X}_{s_1}^{s_1-s_j, x_1}, \ldots,
\mathcal{X}_{s_1}^{s_1-s_j, x_m})-\mathbb{E}_{\mu_p}F(\eta_{s_j})\right)\right)
\]
from above, we further introduce some notations.
For $1\leq l\leq 4$, we denote by $\mathcal{E}_l$ the event that
\[
\tau_{l, k}^{s_1, s_2, s_3, s_4}>s_1
\]
for all $k\in \{1, 2, 3, 4\}\setminus \{l\}$, where $\tau_{l, k}^{s_1, s_2, s_3, s_4}$ is defined as in \eqref{equ 3.21}.

For any $1\leq l<k\leq 4$, we denote by $\mathcal{E}_{lk}$ the event that
\[
\tau_{l, k}^{s_1, s_2, s_3, s_4}\leq s_1.
\]
For $l, k, \hat{l}, \hat{k}$ such that $\{l, k\}\cup \{\hat{l}, \hat{k}\}=\{1, 2, 3, 4\}$ and $\{l, k\}\cap \{\hat{l}, \hat{k}\}=\emptyset$,
we define $\mathcal{E}_{lk, \hat{l}\hat{k}}=\mathcal{E}_{lk}\cap \mathcal{E}_{\hat{l}\hat{k}}$ and
\[
\mathcal{E}_{lk, \hat{l}, \hat{k}}=\mathcal{E}_{lk}\cap\mathcal{E}_{\hat{l}}\cap\mathcal{E}_{\hat{k}}.
\]
For $1\leq l<k<j\leq 4$, we define
\[
\mathcal{E}_{lkj}=\left(\mathcal{E}_{lk}\cap \mathcal{E}_{lj}\right)\cup \left(\mathcal{E}_{lk}\cap \mathcal{E}_{kj}\right)\cup \left(\mathcal{E}_{lj}\cap \mathcal{E}_{kj}\right).
\]

We denote by $\mathcal{E}_{1, 2, 3, 4}$ the event that
\[
\tilde{\tau}^{s_1, s_2, s_3, s_4}>s_1.
\]
For simplicity, we write $(s_1, s_2, s_3, s_4)$ as $\vec{s}$ and write all above events $\mathcal{E}_\cdot$ as $\mathcal{E}_\cdot^{\vec{s}}$ when we need to emphasize the dependence on $s_1, s_2, s_3, s_4$. For simplicity, we denote by $\mathcal{J}^{\vec{s}}$
the random variable
\[
\mathbb{E}_{\mu_p}\left(\prod_{j=1}^4\left(H(\eta_0, \mathcal{X}_{s_1}^{s_1-s_j, x_1}, \ldots,
\mathcal{X}_{s_1}^{s_1-s_j, x_m})-\mathbb{E}_{\mu_p}F(\eta_{s_j})\right)\right).
\]
We bound $\hat{\mathbb{E}}\mathcal{J}^{\vec{s}}$ from above step by step. First we have the following lemma.
\begin{lemma}\label{lemma 3.9}
There exists $C_6<+\infty$ independent of $a, b, B^{\vec{s}}$ such that
\begin{align*}
\int_a^bds_1\int_a^{s_1}ds_2\int_a^{s_2}ds_3\int_a^{s_3}\left|\hat{\mathbb{E}}\left(1_{B^{\vec{s}}}\mathcal{J}^{\vec{s}}\right)\right|ds_4
\leq C_6(b-a)^2
\end{align*}
for any $0<a<b$ and $B^{\vec{s}}=\mathcal{E}_{lk, \hat{l}\hat{k}}^{\vec{s}}$ or $\mathcal{E}_{lkj}^{\vec{s}}$ for any different $l, k, j\in \{1, 2, 3, 4\}$ and $\{\hat{l}, \hat{k}\}=\{1, 2, 3, 4\}\setminus\{l, k\}$.
\end{lemma}

\proof
When the event $B^{\vec{s}}=\mathcal{E}_{lk, \hat{l}\hat{k}}^{\vec{s}}$ occurs for some $l<k, \hat{l}<\hat{k}$, there exists $x, y, z, w\in \{x_1, \ldots, x_m\}$ such that the event
\[
A_{l, k, x, y}^{s_1, s_2, s_3, s_4}\bigcap A_{\hat{l}, \hat{k}, z, w}^{s_1, s_2, s_3, s_4}
\]
occurs. Since $\{l, k\}\cap \{\hat{l}, \hat{k}\}=\emptyset$, the events $A_{l, k, x, y}^{s_1, s_2, s_3, s_4}$ and $A_{\hat{l}, \hat{k}, z, w}^{s_1, s_2, s_3, s_4}$ are independent. Therefore,
by Lemma \ref{lemma 3.7},
\begin{align*}
&\int_a^bds_1\int_a^{s_1}ds_2\int_a^{s_2}ds_3\int_a^{s_3}\hat{\mathbb{P}}\left(A_{l, k, x, y}^{s_1, s_2, s_3, s_4}\bigcap A_{\hat{l}, \hat{k}, z, w}^{s_1, s_2, s_3, s_4}\right)ds_4\\
&\leq \int_a^b (C_3/C_2)ds_l\int_a^b (C_3/C_2)ds_{\hat{l}}=(C_3/C_2)^2(b-a)^2
\end{align*}
and hence
\begin{align*}
&\int_a^bds_1\int_a^{s_1}ds_2\int_a^{s_2}ds_3\int_a^{s_3}\left|\hat{\mathbb{E}}\left(1_{B^{\vec{s}}}\mathbb{E}_{\mu_p}\left(\prod_{j=1}^4H(\eta_0, \mathcal{X}_{s_1}^{s_1-s_j, x_1}, \ldots,
\mathcal{X}_{s_1}^{s_1-s_j, x_m})\right)\right)\right|ds_4\\
&=O(1)(b-a)^2.
\end{align*}

When $B^{\vec{s}}=\mathcal{E}_{lkj}^{\vec{s}}$ occurs, then one of $\mathcal{E}_{lk}\cap \mathcal{E}_{lj}, \mathcal{E}_{lk}\cap \mathcal{E}_{kj}, \mathcal{E}_{lj}\cap \mathcal{E}_{kj}$ occurs. When $\mathcal{E}_{lk}\cap \mathcal{E}_{lj}$ occurs, there exists $x, y, z, w\in \{x_1, \ldots, x_m\}$ such that
\[
A_{l, k, x, y}^{s_1, s_2, s_3, s_4}\bigcap A_{l, j, z, w}^{s_1, s_2, s_3, s_4}
\]
occurs. When $x=z$, then the event $\mathcal{H}_{l, k, j, x, y, w}^{s_1, s_2, s_3, s_4}$ occurs. By Lemma \ref{lemma 3.8},
\begin{align*}
&\int_a^bds_1\int_a^{s_1}ds_2\int_a^{s_2}ds_3\int_a^{s_3}\hat{\mathbb{P}}\left(\mathcal{H}_{l, k, j, x, y, w}^{s_1, s_2, s_3, s_4}\right)ds_4\\
&\leq \int_a^b ds_{\hat{l}}\int_a^b 3(C_3/C_2)^2 ds_{\hat{a}}=3(b-a)^2(C_3/C_2)^2,
\end{align*}
where $\hat{a}=\max\{l, k, j\}$ and $\hat{l}=\{1, 2, 3, 4\}\setminus\{l, k, j\}$.

When $x\neq z$, we denote by $\tilde{\varphi}_1$ the collision moment of $\mathcal{X}_{\cdot, l}^{s_1-s_l, x}$ and $\mathcal{X}_{\cdot, l}^{s_1-s_l, z}$, by $\tilde{\varphi}_2$ the first collision moment of $\mathcal{X}_{\cdot, l}^{s_1-s_l, x}$ and $\mathcal{X}_{\cdot, k}^{s_1-s_k, y}$, by $\tilde{\varphi}_3$ the first collision moment of $\mathcal{X}_{\cdot, l}^{s_1-s_l, z}$ and $\mathcal{X}_{\cdot, j}^{s_1-s_j, w}$. Let

When $\tilde{\varphi}_1<\min\{\tilde{\varphi}_2, \tilde{\varphi}_3\}$, then the event
\[
\mathcal{H}_{l, k, j, x, y, w}^{s_1, s_2, s_3, s_4}\bigcup \mathcal{H}_{l, k, j, z, y, w}^{s_1, s_2, s_3, s_4}
\]
occurs. In detail, when $\mathcal{X}_{\cdot, l}^{s_1-s_l, z}$ (resp. $\mathcal{X}_{\cdot, l}^{s_1-s_l, x}$) is erased after $\tilde{\varphi}_1$, the event
$\mathcal{H}_{l, k, j, x, y, w}^{s_1, s_2, s_3, s_4}$ (resp. $\mathcal{H}_{l, k, j, z, y, w}^{s_1, s_2, s_3, s_4}$) occurs. By Lemma \ref{lemma 3.8},
\begin{align*}
&\int_a^bds_1\int_a^{s_1}ds_2\int_a^{s_2}ds_3\int_a^{s_3}\hat{\mathbb{P}}\left(\mathcal{H}_{l, k, j, x, y, w}^{s_1, s_2, s_3, s_4}\bigcup \mathcal{H}_{l, k, j, z, y, w}^{s_1, s_2, s_3, s_4}\right)ds_4\\
&\leq 6(b-a)^2(C_3/C_2)^2.
\end{align*}
When $\tilde{\varphi}_1>\max\{\tilde{\varphi}_2, \tilde{\varphi}_3\}$, the event
\[
A_{l, k, x, y}^{s_1, s_2, s_3, s_4}\bigcap \tilde{A}_{l, j, z, w}^{s_1, s_2, s_3, s_4}
\]
occurs, where $\tilde{A}_{l, j, z, w}^{s_1, s_2, s_3, s_4}$ is the event that
\[
\tilde{\mathcal{X}}_{t,l}^{s_1-s_l, z}=\mathcal{X}_{t, j}^{s_1-s_j, w}
\]
for some $t\geq \max\{s_1-s_j, s_1-s_l\}$, where $\tilde{\mathcal{X}}_{\cdot, l}^{s_1-s_l, z}$ is a copy of $\mathcal{X}_{\cdot, l}^{s_1-s_l, z}$ such that $\tilde{\mathcal{X}}_{\cdot, l}^{s_1-s_l, z}$ is independent of $\mathcal{X}_{\cdot, l}^{s_1-s_l, x}$ and
\[
\tilde{\mathcal{X}}_{t, l}^{s_1-s_l, z}=\mathcal{X}_{t, l}^{s_1-s_l, z}
\]
for $0\leq t\leq \tilde{\varphi}_1$. Note that $\tilde{A}_{l, j, z, w}^{s_1, s_2, s_3, s_4}$ is independent of $A_{l, k, x, y}^{s_1, s_2, s_3, s_4}$. Therefore, by Lemma \ref{lemma 3.7},
\begin{align*}
&\int_a^bds_1\int_a^{s_1}ds_2\int_a^{s_2}ds_3\int_a^{s_3}\hat{\mathbb{P}}\left(A_{l, k, x, y}^{s_1, s_2, s_3, s_4}\bigcap \tilde{A}_{l, j, z, w}^{s_1, s_2, s_3, s_4}\right)ds_4 \\
&\leq \int_a^bds_{\hat{l}}\int_a^b(C_3/C_2)^2ds_{\hat{a}}=(C_3/C_2)^2(b-a)^2.
\end{align*}
According to the definition of the coalescing random walk, we can further assume that the process $\mathcal{X}_{\cdot, l}^{s_1-s_l, x}$ (resp. $\tilde{\mathcal{X}}_{\cdot, l}^{s_1-s_l, z}$) is remained and $\tilde{\mathcal{X}}_{\cdot, l}^{s_1-s_l, z}$ (resp. $\mathcal{X}_{\cdot, l}^{s_1-s_l, x}$) is erased after $t\geq \tilde{\varphi}_1$ when $x=x_i, z=x_j$ for some $i<j$ (resp. $i>j$). Therefore, when $\tilde{\varphi}_2\leq \tilde{\varphi}_1\leq \tilde{\varphi}_3$, the event
\[
\mathcal{H}_{l, k, j, x, y, w}^{s_1, s_2, s_3, s_4}\bigcup \left(A_{l, k, x, y}^{s_1, s_2, s_3, s_4}\bigcap \tilde{A}_{l, j, z, w}^{s_1, s_2, s_3, s_4}\right)
\]
occurs. In detail, when $\mathcal{X}_{\cdot, l}^{s_1-s_l, z}$ (resp. $\mathcal{X}_{\cdot, l}^{s_1-s_l, x}$) is erased after $\tilde{\varphi}_1$, the event
$\mathcal{H}_{l, k, j, x, y, w}^{s_1, s_2, s_3, s_4}$ (resp. $A_{l, k, x, y}^{s_1, s_2, s_3, s_4}\bigcap \tilde{A}_{l, j, z, w}^{s_1, s_2, s_3, s_4}$) occurs.
Similarly, when $\tilde{\varphi}_3\leq \tilde{\varphi}_1\leq \tilde{\varphi}_2$, the event
\[
\mathcal{H}_{l, k, j, z, y, w}^{s_1, s_2, s_3, s_4}\bigcup \left(A_{l, k, x, y}^{s_1, s_2, s_3, s_4}\bigcap \tilde{A}_{l, j, z, w}^{s_1, s_2, s_3, s_4}\right)
\]
occurs. By Lemmas \ref{lemma 3.7} and \ref{lemma 3.8},
\begin{align*}
&\int_a^bds_1\int_a^{s_1}ds_2\int_a^{s_2}ds_3\int_a^{s_3}\hat{\mathbb{P}}\left(\mathcal{H}_{l, k, j, z, y, w}^{s_1, s_2, s_3, s_4}
\bigcup\left(A_{l, k, x, y}^{s_1, s_2, s_3, s_4}\bigcap \tilde{A}_{l, j, z, w}^{s_1, s_2, s_3, s_4}\right)\right)ds_4 \\
&\leq 4(C_3/C_2)^2(b-a)^2.
\end{align*}
As a result, we have
\begin{align*}
&\int_a^bds_1\int_a^{s_1}ds_2\int_a^{s_2}ds_3\int_a^{s_3}\left|\hat{\mathbb{E}}\left(1_{B^{\vec{s}}}\mathbb{E}_{\mu_p}\left(\prod_{j=1}^4H(\eta_0, \mathcal{X}_{s_1}^{s_1-s_j, x_1}, \ldots,
\mathcal{X}_{s_1}^{s_1-s_j, x_m})\right)\right)\right|ds_4\\
&=O(1)(b-a)^2
\end{align*}
when $B^{\vec{s}}=\mathcal{E}_{lk}^{\vec{s}}\cap \mathcal{E}_{lj}^{\vec{s}}$ and hence when $B^{\vec{s}}=\mathcal{E}_{lkj}^{\vec{s}}$. In conclusion, the proof is complete.
\qed

Then, we have the following lemma.
\begin{lemma}\label{lemma 3.10}
There exists $C_7<+\infty$ independent of $a, b, B^{\vec{s}}$ such that
\begin{align*}
\int_a^bds_1\int_a^{s_1}ds_2\int_a^{s_2}ds_3\int_a^{s_3}\left|\hat{\mathbb{E}}\left(1_{B^{\vec{s}}}\mathcal{J}^{\vec{s}}\right)\right|ds_4
\leq C_7(b-a)^2
\end{align*}
for any $0<a<b$ and $B^{\vec{s}}=\mathcal{E}_{lk, \hat{l}, \hat{k}}^{\vec{s}}$ for any different $l, k\in \{1, 2, 3, 4\}$ and $\{\hat{l}, \hat{k}\}=\{1, 2, 3, 4\}\setminus\{l, k\}$.
\end{lemma}

\proof
We only deal with the case where $l=1, k=2$ since other cases follow from similar arguments. On the event $\mathcal{E}_{12, 3, 4}^{\vec{s}}$,
\begin{align*}
\mathcal{J}^{\vec{s}}
&=\mathbb{E}_{\mu_p}\left(\prod_{j=1}^2\left(H(\eta_0, \mathcal{X}_{s_1}^{s_1-s_j, x_1}, \ldots,
\mathcal{X}_{s_1}^{s_1-s_j, x_m})-\mathbb{E}_{\mu_p}F(\eta_{s_j})\right)\right)
\\
&\text{\quad}\times\prod_{j=3}^4\mathbb{E}_{\mu_p}\left(H(\eta_0, \mathcal{X}_{s_1, j}^{s_1-s_j, x_1}, \ldots,
\mathcal{X}_{s_1, j}^{s_1-s_j, x_m})-\mathbb{E}_{\mu_p}F(\eta_{s_j})\right)
\end{align*}
and there exists a random variable $\mathcal{U}^{\vec{s}}$ measurable with respect to
\[
\left\{\mathcal{X}_{r, 1}^{0, x_1}, \ldots,
\mathcal{X}_{r, 1}^{0, x_m}, \mathcal{X}_{r, 2}^{s_1-s_2, x_1}, \ldots, \mathcal{X}_{r, 2}^{s_1-s_2, x_m}\right\}_{0\leq r\leq s_1}
\]
such that
\[
\mathbb{E}_{\mu_p}\left(\prod_{j=1}^2\left(H(\eta_0, \mathcal{X}_{s_1}^{s_1-s_j, x_1}, \ldots,
\mathcal{X}_{s_1}^{s_1-s_j, x_m})-\mathbb{E}_{\mu_p}F(\eta_{s_j})\right)\right)=\mathcal{U}^{\vec{s}}
\]
on the event $\mathcal{E}_{12, 3, 4}^{\vec{s}}$.

For simplicity, we denote by $\mathcal{R}^{\vec{s}}$ the term
\[
\mathcal{U}^{\vec{s}}\prod_{j=3}^4\mathbb{E}_{\mu_p}\left(H(\eta_0, \mathcal{X}_{s_1, j}^{s_1-s_j, x_1}, \ldots,
\mathcal{X}_{s_1, j}^{s_1-s_j, x_m})-\mathbb{E}_{\mu_p}F(\eta_{s_j})\right).
\]
We denote by $\mathcal{C}^{\vec{s}}$ the set
\[
\mathcal{E}_{12}^{\vec{s}}\setminus \mathcal{E}_{12, 3, 4}^{\vec{s}},
\]
then
\[
\mathcal{C}^{\vec{s}}\subseteq \mathcal{E}_{12, 34}^{\vec{s}}\bigcup \mathcal{E}_{123}^{\vec{s}} \bigcup \mathcal{E}_{124}^{\vec{s}}.
\]
Therefore, according to an argument similar to that leading to Lemma \ref{lemma 3.9},
\[
\int_a^bds_1\int_a^{s_1}ds_2\int_a^{s_2}ds_3\int_a^{s_3}\left|\hat{\mathbb{E}}\left(1_{\mathcal{C}^{\vec{s}}}\mathcal{R}^{\vec{s}}\right)\right|ds_4
=O(1)(b-a)^2.
\]
Hence, to complete the proof, we only need to show that
\begin{equation*}
\int_a^bds_1\int_a^{s_1}ds_2\int_a^{s_2}ds_3\int_a^{s_3}\left|\hat{\mathbb{E}}\left(1_{\mathcal{E}_{12}^{\vec{s}}}\mathcal{R}^{\vec{s}}\right)\right|ds_4
=O(1)(b-a)^2.
\end{equation*}
Actually, we have a stronger conclusion that
\begin{equation}\label{equ 3.22}
\hat{\mathbb{E}}\left(1_{\mathcal{E}_{12}^{\vec{s}}}\mathcal{R}^{\vec{s}}\right)=0.
\end{equation}
Note that $1_{\mathcal{E}_{12}^{\vec{s}}}\mathcal{U}^{\vec{s}}$
is measurable with respect to
\[
\mathcal{X}_{\cdot, 1}^{0, x_1}, \ldots,
\mathcal{X}_{\cdot, 1}^{0, x_m}, \mathcal{X}_{\cdot, 2}^{s_1-s_2, x_1}, \ldots, \mathcal{X}_{\cdot, 2}^{s_1-s_2, x_m}
\]
and hence is independent of
\[
\mathcal{X}_{s_1, 3}^{s_1-s_3, x_1}, \ldots,
\mathcal{X}_{s_1, 3}^{s_1-s_3, x_m}, \mathcal{X}_{s_1, 4}^{s_1-s_4, x_1}, \ldots, \mathcal{X}_{s_1, 4}^{s_1-s_4, x_m}.
\]
Therefore, by Proposition \ref{proposition 3.2 dual equivalent},
\begin{align*}
\hat{\mathbb{E}}\left(1_{\mathcal{E}_{12}^{\vec{s}}}\mathcal{R}^{\vec{s}}\right)
&=\hat{\mathbb{E}}\left(1_{\mathcal{E}_{12}^{\vec{s}}}\mathcal{U}^{\vec{s}}\right)
\prod_{j=3}^4\hat{\mathbb{E}}\mathbb{E}_{\mu_p}\left(H(\eta_0, \mathcal{X}_{s_1, j}^{s_1-s_j, x_1}, \ldots,
\mathcal{X}_{s_1, j}^{s_1-s_j, x_m})-\mathbb{E}_{\mu_p}F(\eta_{s_j})\right)\\
&=\hat{\mathbb{E}}\left(1_{\mathcal{E}_{12}^{\vec{s}}}\mathcal{U}^{\vec{s}}\right)
\prod_{j=3}^4\left(\mathbb{E}_{\mu_p}F(\eta_{s_j})-\mathbb{E}_{\mu_p}F(\eta_{s_j})\right)=0.
\end{align*}
Since \eqref{equ 3.22} holds, the proof is complete.
\qed

Next, we have the following lemma.
\begin{lemma}\label{lemma 3.11}
There exists $C_8<+\infty$ independent of $a, b$ such that
\begin{align*}
\int_a^bds_1\int_a^{s_1}ds_2\int_a^{s_2}ds_3\int_a^{s_3}\left|\hat{\mathbb{E}}\left(1_{\mathcal{E}_{1, 2, 3, 4}^{\vec{s}}}\mathcal{J}^{\vec{s}}\right)\right|ds_4
\leq C_8(b-a)^2
\end{align*}
for any $0<a<b$.
\end{lemma}
\proof
On the event $\mathcal{E}_{1, 2, 3, 4}$,
\[
\mathcal{J}^{\vec{s}}=\prod_{j=1}^4\mathbb{E}_{\mu_p}\left(H(\eta_0, \mathcal{X}_{s_1, j}^{s_1-s_j, x_1}, \ldots,
\mathcal{X}_{s_1, j}^{s_1-s_j, x_m})-\mathbb{E}_{\mu_p}F(\eta_{s_j})\right).
\]
For simplicity, we denote by $\mathcal{V}^{\vec{s}}$ the term
\[
\prod_{j=1}^4\mathbb{E}_{\mu_p}\left(H(\eta_0, \mathcal{X}_{s_1, j}^{s_1-s_j, x_1}, \ldots,
\mathcal{X}_{s_1, j}^{s_1-s_j, x_m})-\mathbb{E}_{\mu_p}F(\eta_{s_j})\right).
\]
Let $(\mathcal{E}_{1, 2, 3, 4}^{\vec{s}})^c$ be the complement of $\mathcal{E}_{1, 2, 3, 4}^{\vec{s}}$, then
\[
(\mathcal{E}_{1, 2, 3, 4}^{\vec{s}})^c
\subseteq \bigcup_{l,k,j}\mathcal{E}_{lkj}\bigcup\bigcup_{l, k}\mathcal{E}_{lk, \hat{l}\hat{k}}\bigcup\bigcup_{l, k}\mathcal{E}_{lk, \hat{l}, \hat{k}}.
\]
According to an argument similar to that leading to Lemma \ref{lemma 3.9},
\[
\int_a^bds_1\int_a^{s_1}ds_2\int_a^{s_2}ds_3\int_a^{s_3}\left|\hat{\mathbb{E}}\left(1_{\mathcal{E}^{\vec{s}}_{lk, \hat{l}\hat{k}}}\mathcal{V}^{\vec{s}}\right)\right|ds_4=O(1)(b-a)^2
\]
and
\[
\int_a^bds_1\int_a^{s_1}ds_2\int_a^{s_2}ds_3\int_a^{s_3}\left|\hat{\mathbb{E}}\left(1_{\mathcal{E}_{lkj}^{\vec{s}}}\mathcal{V}^{\vec{s}}\right)\right|ds_4
=O(1)(b-a)^2.
\]
According to an argument similar to that leading to Lemma \ref{lemma 3.10},
\[
\int_a^bds_1\int_a^{s_1}ds_2\int_a^{s_2}ds_3\int_a^{s_3}\left|\hat{\mathbb{E}}\left(1_{\mathcal{E}_{lk, \hat{l}, \hat{k}}^{\vec{s}}}\mathcal{V}^{\vec{s}}\right)\right|ds_4=O(1)(b-a)^2.
\]
Therefore,
\[
\int_a^bds_1\int_a^{s_1}ds_2\int_a^{s_2}ds_3\int_a^{s_3}\left|\hat{\mathbb{E}}\left(1_{(\mathcal{E}_{1, 2, 3, 4}^{\vec{s}})^c}\mathcal{V}^{\vec{s}}\right)\right|ds_4=O(1)(b-a)^2.
\]
Hence, to complete this proof, we only need to show that
\[
\int_a^bds_1\int_a^{s_1}ds_2\int_a^{s_2}ds_3\int_a^{s_3}\left|\hat{\mathbb{E}}\left(\mathcal{V}^{\vec{s}}\right)\right|ds_4=O(1)(b-a)^2.
\]
Actually, according to an analysis similar to that leading to \eqref{equ 3.22}, we have the stronger conclusion that
\[
\int_a^bds_1\int_a^{s_1}ds_2\int_a^{s_2}ds_3\int_a^{s_3}\left|\hat{\mathbb{E}}\left(\mathcal{V}^{\vec{s}}\right)\right|ds_4=0
\]
and hence the proof is complete.
\qed

Now we prove Theorem \ref{theorem 2.3 main} in the case where $\eta_0$ is distributed with $\nu_p$.
\proof[Proof of Theorem \ref{theorem 2.3 main} in the case where $\eta_0$ is distributed with $\nu_p$]
Since
\[
\mathcal{E}_{1, 2, 3, 4}\bigcup\bigcup_{l,k,j}\mathcal{E}_{lkj}\bigcup\bigcup_{l, k}\mathcal{E}_{lk, \hat{l}\hat{k}}\bigcup\bigcup_{l, k}\mathcal{E}_{lk, \hat{l}, \hat{k}}
\]
is the universal set, according to Lemmas \ref{lemma 3.9}-\ref{lemma 3.11},
\begin{equation}\label{equ 3.23 two}
\int_{sN}^{tN}ds_1\int_{sN}^{s_1}ds_2\int_{sN}^{s_2}ds_3\int_{sN}^{s_3}\left|\hat{\mathbb{E}}\left(\mathcal{J}^{\theta+\vec{s}}\right)\right|ds_4
\leq C_9(t-s)^2N^2
\end{equation}
for any $0\leq s\leq t$ and $\theta>0$, where $\theta+\vec{s}=(s_1+\theta, s_2+\theta, s_3+\theta, s_4+\theta)$ and $C_9<+\infty$ is a constant independent of $s, t, N, \theta$. Let $\theta\rightarrow+\infty$ in \eqref{equ 3.23 two}, then by Proposition \ref{proposition 3.2 dual equivalent},
\begin{equation}\label{equ 3.23}
\mathbb{E}_{\nu_p}\left(\left(\frac{1}{\sqrt{N}}\xi_{tN}^F-\frac{1}{\sqrt{N}}\xi_{sN}^F\right)^4\right)\leq 24C_9(t-s)^2
\end{equation}
for any $0\leq s\leq t$. By \eqref{equ 3.23} and Corollary 14.9 of \cite{Kallenberg1997}, the sequence
\[
\left\{\frac{1}{\sqrt{N}}\xi_{tN}^F:~0\leq t\leq T\right\}_{N\geq 1}
\]
are tight with respect to the uniform topology of $C[0, T]$. Theorem \ref{theorem 2.3 main} follows from the aforesaid tightness and Corollary \ref{corollary 3.6}.
\qed


At last, we give an outline of how to prove Theorem \ref{theorem 2.3 main} when $\eta_0$ is distributed with $\mu_p$. Note that, from now on
$\xi_t^F=\int_0^t\left(F(\eta_s)-\mathbb{E}_{\mu_p}F(\eta_s)\right)ds$. Let $G^N$ be defined as in \eqref{equ 3.11 two}. We further define
\[
\tilde{G}^N(t, \eta)=G^N(\eta)-\mathbb{E}_{\mu_p}G^N(\eta_t).
\]
By \eqref{equ 3.11 two},
\[
M_t^N+C_{10}(t)=\tilde{G}^N(t, \eta_t)-\tilde{G}^N(0, \eta_0)-\int_0^t\left(S(N)F(\eta_s)-\mathbb{E}_{\mu_p}S(N)F(\eta_s)\right)ds+\xi_t^F,
\]
where
\[
C_{10}(t)=-\mathbb{E}_{\mu_p}G^N(\eta_t)+\mathbb{E}_{\mu_p}G^N(\eta_0)+\int_0^t\mathbb{E}_{\mu_p}S(N)F(\eta_s)ds-\int_0^t\mathbb{E}_{\mu_p}F(\eta_s)ds.
\]
Since the means of $M_t^N, \tilde{G}^N(t, \eta_t), S(N)F(\eta_s)-\mathbb{E}_{\mu_p}S(N)F(\eta_s), \xi_t^F$ under $\mu_p$ are all $0$, we have $C_{10}(t)=0$ and hence
\begin{align}\label{equ 3.25}
\frac{1}{\sqrt{N}}M_{tN}^N=&\frac{1}{\sqrt{N}}\tilde{G}^N(tN, \eta_{tN})-\frac{1}{\sqrt{N}}\tilde{G}^N(0, \eta_0)\notag\\
&-\frac{1}{\sqrt{N}}
\int_0^{tN}\left(S(N)F(\eta_s)-\mathbb{E}_{\mu_p}S(N)F(\eta_s)\right)ds+\frac{1}{\sqrt{N}}\xi_{tN}^F,
\end{align}
which is an analogue of \eqref{equ 3.12}. We have the following analogue of Lemma \ref{lemma 3.4}.

\begin{lemma}\label{lemma 3.12}
Let $\eta_0$ be distributed with $\mu_p$.
For any $t\geq 0$,
\begin{align*}
&\lim_{N\rightarrow+\infty}\Big(\frac{1}{\sqrt{N}}\tilde{G}^N(tN, \eta_{tN})-\frac{1}{\sqrt{N}}\tilde{G}^N(0, \eta_0)\\
&\text{\quad\quad\quad}-\frac{1}{\sqrt{N}}
\int_0^{tN}\left(S(N)F(\eta_s)-\mathbb{E}_{\mu_p}S(N)F(\eta_s)\right)ds\Big)=0
\end{align*}
in $L^2$.
\end{lemma}

\proof
According to an argument similar to that given in the proof of Lemma \ref{lemma 3.4}, we have
\begin{align*}
{\rm Var}_{\mu_p}\left(\frac{1}{\sqrt{N}}\tilde{G}^N(t, \eta_t)\right)\leq \frac{2\|H\|_\infty^2}{NC_2}\sum_{i=1}^m\sum_{j=1}^m\int_0^Nds_1\left(\int_0^{+\infty}\int_0^{+\infty}p_{s_1+s_2+\theta}(x_i, x_j)ds_2d\theta\right)
\end{align*}
and
\begin{align*}
{\rm Var}_{\mu_p}\left(\frac{1}{\sqrt{N}}\int_0^tS(N)F(\eta_s)ds\right)
\leq \frac{4\|H\|_\infty^2}{C_2}\sum_{i=1}^m\sum_{j=1}^m\int_0^{+\infty}\int_0^{+\infty}p_{2N+v+s_2}(x_i, x_j)dvds_2
\end{align*}
for any $t\geq 0$. Lemma \ref{lemma 3.12} follows from the above two inequalities and Assumption \eqref{Assumption 1.6}.
\qed

Then we have the following analogue of Lemma \ref{lemma 3.5}.
\begin{lemma}\label{lemma 3.13}
Let $T>0$ and $\eta_0$ be distributed with $\mu_p$. As $N\rightarrow+\infty$, the process
\[
\left\{\frac{1}{\sqrt{N}}M_{tN}^N:~0\leq t\leq T\right\}
\]
converges weakly, with respect to the Skorohod topology of $D[0, T]$, to $\{\sigma_FB_t\}_{0\leq t\leq T}$.
\end{lemma}

\proof
As in the proof of Lemma \ref{lemma 3.5}, we only need to show that
\begin{equation}\label{equ 3.26}
\lim_{N\rightarrow+\infty}\frac{1}{N}\mathbb{E}_{\mu_p}\langle M^N\rangle_{tN}=\sigma_F^2 t
\end{equation}
and
\begin{equation}\label{equ 3.27}
\lim_{N\rightarrow+\infty}{\rm Var}_{\mu_p}\left(\frac{1}{N}\langle M^N\rangle_{tN}\right)=0.
\end{equation}
Equation \eqref{equ 3.26} follows from Equation \eqref{equ 3.25},  Lemmas \ref{lemma 3.3}, \ref{lemma 3.12} and the fact that
\[
\mathbb{E}_{\mu_p}\langle M^N\rangle_{tN}=\mathbb{E}_{\mu_p}\left(\left(M^N_{tN}\right)^2\right).
\]
Equation \eqref{equ 3.27} follows from an argument similar to that leading to \eqref{equ 3.18}.
\qed

By Lemmas \ref{lemma 3.12}, \ref{lemma 3.13} and Equation \eqref{equ 3.25}, we have the following analogue of Corollary \ref{corollary 3.6}.

\begin{corollary}\label{corollary 3.14}
Let $\eta_0$ be distributed with $\mu_p$. For any $0<t_1<t_2<\ldots<t_k$, the $\mathbb{R}^k$-valued random variable
\[
\left(\frac{1}{\sqrt{N}}\xi_{t_1N}^F, \frac{1}{\sqrt{N}}\xi_{t_2N}^F, \ldots, \frac{1}{\sqrt{N}}\xi_{t_kN}^F\right)
\]
converges weakly to $\sigma_F\left(B_{t_1}, B_{t_2}, \ldots, B_{t_k}\right)$ as $N\rightarrow+\infty$.
\end{corollary}

Now we prove Theorem \ref{theorem 2.3 main} in the case where $\eta_0$ is distributed with $\mu_p$.
\proof[Proof of Theorem \ref{theorem 2.3 main} in the case where $\eta_0$ is distributed with $\mu_p$]
By Equation \eqref{equ 3.23 two},
\begin{equation*}
\mathbb{E}_{\mu_p}\left(\left(\frac{1}{\sqrt{N}}\xi_{tN}^F-\frac{1}{\sqrt{N}}\xi_{sN}^F\right)^4\right)\leq 24C_9(t-s)^2
\end{equation*}
for any $0\leq s<t$. Consequently, when $\eta_0$ is distributed with $\mu_p$, the sequence
\[
\left\{\frac{1}{\sqrt{N}}\xi_{tN}^F:~0\leq t\leq T\right\}_{N\geq 1}
\]
are tight with respect to the uniform topology of $C[0, T]$. Theorem \ref{theorem 2.3 main} follows from the aforesaid tightness and Corollary \ref{corollary 3.14}.
\qed

\quad

\textbf{Acknowledgments.}
The author is grateful to financial
supports from the National Natural Science Foundation of China with grant number 12371142.

{}
\end{document}